\documentclass[12pt,a4paper,reqno]{amsart}%[a4paper]
\usepackage{latexsym,amsfonts,amssymb,epsfig}
\usepackage{hyperref} %% autoreferences, works only with PdfLaTeX
\usepackage{color}
\usepackage{graphicx}
\usepackage[pagewise]{lineno} %\linenumbers  %% line numbering

\newtheorem{theorem}{Theorem}
\newtheorem{lemma}[theorem]{Lemma}%[section]
\newtheorem{corollary}[theorem]{Corollary}%[section]
\theoremstyle{Conjecture}
\newtheorem{conjecture}{Conjecture}
\newtheorem{Remark}{Remark}

\DeclareMathOperator{\ch}{char}
\DeclareMathOperator{\Der}{Der}
\DeclareMathOperator{\Lie}{Lie}
\DeclareMathOperator{\Alg}{Alg}
\DeclareMathOperator{\End}{End}
\DeclareMathOperator{\Ker}{Ker}
\DeclareMathOperator{\Imm}{Im}
\DeclareMathOperator{\wt}{wt}    % weight
\DeclareMathOperator{\swt}{swt}  % superweight
\DeclareMathOperator{\Wt}{Wt}    % vector weight function
\DeclareMathOperator{\GKdim}{GKdim}
\DeclareMathOperator{\LGKdim}{\underline{GKdim}}
\DeclareMathOperator{\Gr}{Gr}
\DeclareMathOperator{\gr}{gr}

\renewcommand {\limsup}{\operatorname* {\overline{lim}}}

\DeclareMathOperator{\Dim}{Dim}
\DeclareMathOperator{\LDim}{\underline{Dim}}

\newcommand{\dd}{\partial}
\newcommand{\N}{\mathbb N}            % natural numbers
\newcommand{\NO}{\mathbb N_0}            % natural numbers
\newcommand{\Z}{\mathbb Z}            % integers
\newcommand{\R}{\mathbb R}            % reals
\newcommand{\C}{\mathbb C}            % complex numbers

\begin{document}
\title{Fibonacci Lie algebra revisited}
\author{Victor Petrogradsky}
\address{Department of Mathematics, University of Brasilia, 70910-900 Brasilia DF, Brazil} 
\email{petrogradsky@rambler.ru}
\thanks{The author was partially supported by grant 
Funda\c{c}\~ao de Amparo \`a Pesquisa do Estado de S\~ao Paulo, FAPESP 2022/10889-4.}
 % Graded Lie (super)algebras
%*****************************************************************************************************************
\begin{abstract}
We describe old and prove new results on properties of the Fibonacci Lie algebra in a self-contained exposition.
First, we study the growth of this algebra in more details. 
So, we show that the polynomial behaviour of the growth function in not uniform.
We establish bounds on the growth of its universal enveloping algebra.
We find bounds on nilpotency indices for elements of the Fibonacci restricted Lie algebra.
We prove that the Fibonacci Lie algebra is not PI. 
Our approach is also based on geometric ideas. 
The Fibonacci Lie algebra is $\mathbb Z^2$-graded and its homogeneous components belong to a strip.
We illustrate the results by computing the initial components and show positions of the homogeneous components in the strip.  
We also discuss properties and conjectures on related associative algebras and (restricted) Poisson algebras.
Second, we prove infiniteness results on homology and Euler characteristic of arbitrary finitely generated graded Lie algebras of subexponential growth. 
In case of the Fibonacci Lie algebra,
we find bounds on the homology groups and the Euler characteristic, and determine respective positions on plane. 
The computation of the initial part of the Euler characteristic of the Fibonacci Lie algebra shows that its behaviour is really chaotic.    
Finally, we formulate results and conjectures on homology groups and infinite presentation.
\end{abstract}
\maketitle

%************************************************************************************************************
\section{Introduction}

The Grigorchuk group~\cite{Grigorchuk80} and Gupta-Sidki group~\cite{GuptaSidki83} provide interesting examples of self-similar periodic groups. 
The author constructed their natural analogue in the world of restricted Lie algebras~\cite{Pe06}, that we call the {\it Fibonacci Lie algebra}.
The goal of the paper is to present its known and new properties in a uniform and independent exposition.
We also establish infiniteness results on homology and Euler characteristic of finitely generated graded Lie algebras of subexponential growth 
(see Section~\ref{SHsubexp} and Theorem~\ref{THsubexp}),
apply it to the infinite dimensional simple Cartan type Lie algebras, 
and consider respective results for the Fibonacci Lie algebra.

%***********************************************************
\subsection{Growth of algebras}
By $K$ denote the ground field.
Let $A$ be an algebra generated by a finite set~$X$.
By $\gamma_A(n)=\gamma_A(X,n)$ denote the {\em growth function}, i.e.
$\gamma_A(n)$ is equal to dimension  of the vector space spanned by all  monomials  in $X$ of length not  exceeding  $n$.
%We consider one more growth function $\lambda_A(n)=\gamma_A(n)-\gamma_A(n-1)$, where $n\in \mathbb N$ (we assume that $\gamma_A(0)=0$).

Recall that the free finitely generated associative algebras and free Lie (super)algebras have the exponential growth~\cite{Ba,BMPZ}.
But there are intermediate growths faster than any function $\exp(n^\beta)$, $\beta< 1$, but which are still subexponential.
For example, A.~Lichtman found that finitely generated solvable Lie algebras have the subexponential growth~\cite{Licht84}.

In order to describe such growths the author suggested the following scale of functions~\cite{Pe93}.
Consider the sequence of functions of a natural argument
$\Phi^q_\alpha (n)$, $q=1,2,3,\dots$, $n\in \mathbb N$ with the real parameter $\alpha\in\mathbb R^+$:
\begin{align}
 \Phi^1_\alpha(n)&=\alpha,                       &&q=1, \nonumber \\
 \Phi^2_\alpha(n)&=n^\alpha,                     &&q=2, \nonumber\\
 \Phi^3_\alpha(n)&=\exp(n^{\alpha/(\alpha+1)}),      &&q=3, \label{scale1}\\
 \Phi^q_\alpha(n)&=\exp\left(\frac n {(\ln^{(q-3)}n)^{1/\alpha}} \right); \nonumber
                                        &&q=4,5,\dots.
\end{align}
where $\ln^{(1)}n=\ln n$, $\ln^{(q+1)}n=\ln(\ln^{(q)}n)$, $q\ge 1$.

Consider an algebra $A$ generated by a finite set $X$.
We define the {\em  (upper) dimension of level} $q$, $q=1,2,3,\dots$, and {\em the lower dimension of level} $q$ as
\begin{align*}
\Dim^q A&= \inf \{ \alpha\in\mathbb R^+
     \mid\exists N:\ \gamma_A(n)\le \Phi^q_\alpha(n),\ n\ge N\}, & q&\in \mathbb N;\\
\LDim^q A&= \sup \{ \alpha\in\mathbb R^+
     \mid\exists N:\ \gamma_A(n)\ge \Phi^q_\alpha(n),\ n\ge N\}, & q&\in \mathbb N.
\end{align*}

These $q$-dimensions generalize the Gelfand-Kirillov dimensions,
they are designed to study different types of the subexponential growth, their basic properties are as follows.
\begin{lemma}[{\cite{Pe96}}]\label{L1}
The $q$-dimensions have the following properties:
\begin{itemize}
 \item
   the functions $\Phi^q_\alpha(n)$ have the subexponential growth;
 \item
   the $q$-dimensions do not depend on the generating set $X$;
 \item
   $\Dim^q A=\alpha$ means that $\gamma_A(n)$ behaves like $\Phi^q_\alpha (n)$;
 \item
    $\Dim^1A=\dim_K A$ (vector space dimension);
 \item
    $\Dim^2A=\GKdim A$ (the Gelfand-Kirillov dimension);
 \item
    $\LDim^2A=\LGKdim A$ (the lower Gelfand-Kirillov dimension);
 \item
    $\Dim^3A$, $\LDim^3A$ correspond to superdimensions of Borho and Kraft~\cite{BorKra} upto normalization.  
 \item
   Assume that $\Dim^q A=\alpha$, where $0<\alpha<\infty$ and $q\ge 2$; then\\
   $\Dim^b A=\infty$ for all $b<q$ and  $\Dim^b A=0$ for all $b>q$.
\end{itemize}
\end{lemma}
These dimensions were introduced to study the subexponential growth of solvable Lie algebras and groups~\cite{Pe93,Pe96,Pe99int}.

%*********************************************************************************************
\section{Pivot elements and their relations}
We study so called {\it Fibonacci Lie algebra} $\mathcal L$,
which is a natural analogue of the Grigorchuk group~\cite{Grigorchuk80}.
We study its properties continuing the research~\cite{Pe06,PeSh09,PeSh13fib}.
We start with its definition and elementary properties.

Fix notations. $\N_0:=\N\cup\{0\}$.
Let $X$ be a set in a Lie algebra $L$, then by
$\Lie(X)$ (or $\Lie_p(X)$) denote the (restricted) Lie subalgebra generated by $X$.
Similarly, $\Alg(X)$ denotes the associative subalgebra generated by $X$.
The commutators are left-normed: $[w_1,w_2,\ldots,w_k]:=[\ldots [w_1,w_2],\ldots, w_k]$.
Denote also $[u,x^k]:=[u,x^{k-1},x]$, $k> 2$. In case of a restricted Lie algebra, $\ch K=p$,
the notation $[u,x^p]=[u,x^{[p]}]$ coincides with a commutator with the formal $p$-th power $x^{[p]}$, which will be simply denoted as $x^p$.
For necessary definitions and properties of restricted Lie algebras
we refer the reader to~\cite{JacLie,StrFar,BMPZ}.
\subsection{Pivot elements}\label{SSpivot}
Temporarily, let  $K$ be an {\bf arbitrary field} and consider the polynomial ring $R=K[t_i| i\ge 0 ]$.
Let $\dd_i=\frac {\dd}{\partial t_i}\in \Der R$,  $i\ge 0$, be the respective partial derivations.
Define two derivations:
\begin{align*}
v_1 & :=\dd_1+t_0(\dd_2+t_1(\dd_3+t_2(\dd_4+t_3(\dd_5+t_4(\dd_6+\cdots )))));\\
v_2 & :=\qquad\quad\;\,
\dd_2+t_1(\dd_3+t_2(\dd_4+t_3(\dd_5+t_4(\dd_6+\cdots )))).
\end{align*}
The action on $R$ and arbitrary commutators of such operators are well-defined;
these operators are so called {\em special derivations}, see~\cite{Rad86,PeSh09,PeRaSh,Razmyslov}.
Denote by $\tau:R\to R$ the {\it shift endomorphism}:
$\tau(t_i)=t_{i+1}$ for $i\ge 0$.
Let also $\tau(\dd_i)=\dd_{i+1}$, $i\ge 1$.
It is convenient to write recursively:
\begin{align*}
v_1 & =\dd_1+t_0 \tau(v_1);\\
v_2 & =\tau(v_1).
\end{align*}
Similarly, define
\begin{equation}
\begin{split}
v_i:=\tau^{i-1}(v_1)&=\dd_i+t_{i-1}(\dd_{i+1}+t_i(\dd_{i+2}+ t_{i+1}(\dd_{i+3}+\cdots )))\\
&=\dd_i+t_{i-1} v_{i+1}, \quad\qquad i\ge 1.
\label{vii}
\end{split}
\end{equation}

We refer to $\{v_i\mid i\ge 1\}$ as the {\it pivot elements}.

\subsection{Basic relations}\label{SSbasic}
Commutators of neighbours resemble the Fibonacci relation.
\begin{lemma}\label{Lsosed}
$[v_i,v_{i+1}]=v_{i+2}$ for  $i=1,2,\dots$.
\end{lemma}
\begin{proof} Using~\eqref{vii},
\begin{equation*}
[v_i,v_{i+1}]=[\dd_i+t_{i-1}v_{i+1},v_{i+1}]=[\dd_i, v_{i+1}]=
[\dd_i,\dd_{i+1}+t_iv_{i+2}]=v_{i+2}.\qedhere
\end{equation*}
\end{proof}
We make a convention that a product with a list of increasing indices, like $t_it_{i+1}\cdots t_{j-1}t_j$ is equal to 1 in case $i>j$.
Let us compute commutators  of arbitrary pivot elements.
\begin{lemma}\label{Lcommutators}
Let $K$ be an arbitrary field. Then
$$ [v_i,v_{j}]=t_{i-1}t_{i}\cdots t_{j-3} v_{j+1},\qquad  1\le i<j.$$
\end{lemma}
\begin{proof} Using~\eqref{vii}, we get
\begin{align*}
v_i&=\dd_i+ t_{i-1}\dd_{i+1}+ t_{i-1}t_{i}\dd_{i+2}+\cdots+ t_{i-1}\cdots t_{j-3}\dd_{j-1}
+ t_{i-1}\cdots t_{j-2}v_j;\\
[v_i,v_j]&=[\dd_i+ t_{i-1}\dd_{i+1}+\cdots+ t_{i-1}\cdots t_{j-3}\dd_{j-1}+ t_{i-1}\cdots t_{j-2}v_j,\  v_j]\\
&=[\dd_i+ t_{i-1}\dd_{i+1}+\cdots+ t_{i-1}\cdots t_{j-3}\dd_{j-1},\ \dd_j+t_{j-1}v_{j+1}]\\
&=t_{i-1}\cdots t_{j-3}\dd_{j-1}(t_{j-1})v_{j+1}=t_{i-1}\cdots t_{j-3}v_{j+1}.
\qedhere
\end{align*}
\end{proof}

\begin{lemma}\label{Laction}
For all $n\ge 1$, $j\ge 0$ we have the action
$$v_n(t_j)= \begin{cases}
                    t_{n-1}t_n\cdots t_{j-2}, &  n<j, \\
                    1,                        &  n=j, \\
                    0,                        &  n>j.
                 \end{cases}
$$
\end{lemma}
\begin{proof} We consider the action of the infinite sum~\eqref{vii} on $t_j$.
\end{proof}

Let $\mathcal L=\Lie(v_1,v_2)$ be the Lie subalgebra of $\Der R$ generated by $v_1,v_2$.
\begin{lemma} Define Lie algebras $L_{(i)}:=\Lie(v_i,v_{i+1})$ generated by $v_i,v_{i+1}$ for $i\ge 1$. Then
\begin{enumerate}
\item $L_{(i)}\cong \mathcal L$, $i\ge 1$.
\item We get an infinite chain of isomorphic subalgebras, inclusions being proper
 $$\mathcal L=L_{(1)}\supset L_{(2)}\supset \cdots  \supset L_{(n)}\supset \cdots $$
\item $\mathcal L$ is infinite dimensional.
\end{enumerate}
\end{lemma}
\begin{proof}
Lemma~\ref{Lcommutators} implies that the shift mapping $\tau$ is a monomorphism.
Since $\tau^{i-1}(v_1)=v_i$, $\tau^{i-1}(v_2)=v_{i+1}$,
we get $L_{(i)}= \tau^{i-1} (\mathcal L)$. Also, $v_{i-1}\notin L_{(i)}$.
\end{proof}

%*****************************************************************************
\section{Weight functions and $\mathbf Z^2$-gradings}

In this section we introduce weight functions and establish $\mathbb Z^2$-gradings by multidegree in the generators in a general setting,  in particular, the field is arbitrary.
\subsection{Weight functions}
The field $K$ is still arbitrary.
Let us construct a {\it weight function} defined on all variables and respective derivations in a correlated way, namely we assume that
$wt(\dd_i)=-wt(t_i)=a_i$, $i\ge 0$, the latter being arbitrary complex numbers.
We extend this values to finite products (of any type)  of $\dd_i$ and $t_j$.
Observe that a weight function is additive due to the correlation.
Namely, let $a,b$ be monomials as above, then $wt(a\cdot b)=wt([a,b])=wt(a)+wt(b)$.
We say that not a necessarily finite linear combination of monomials is {\it homogeneous} provided
that all monomials have the same weight.
The additivity extends to Lie or associative products of homogeneous elements as well.
We warn that a weight function is not defined for the zero element.

We want all summands in~\eqref{vii} to have the same weight, so we assume that
$$a_i=\wt v_i=\wt \dd_i= \wt t_{i-1}+\wt v_{i+1}=-a_{i-1}+a_{i+1},\qquad i\ge 1.$$
We get the Fibonacci recurrence relation $a_{i+1}=a_{i}+a_{i-1}$, for $i\ge 1$.
Its two basic solutions yield the {\it weight} and {\it superweight} functions:
  \begin{equation}\label{weights}
  \begin{split}
  \wt v_n &=\wt \dd_n =-\wt t_n=\lambda^{n}, \qquad\quad \lambda:=\frac {1+\sqrt 5}2;       \\
  \swt v_n&=\swt \dd_n=-\swt t_n=\bar{\lambda}^{n},\quad\quad \bar\lambda:=\frac{1-\sqrt 5}2;
  \end{split}\qquad\quad n\in\mathbb N.
  \end{equation}
(in~\cite{PeSh09,PeSh13fib} we used a different setting $\swt v_n=\bar{\lambda}^{n-2}$).
We use the following properties without mentioning:
\begin{equation}\label{lambd}
\lambda\cdot \bar \lambda=-1,\qquad \lambda^2=\lambda+1,\qquad \frac 1\lambda =\lambda-1.
\end{equation}
We define the {\it vector weight function} for homogeneous elements:
\begin{equation} \label{Wt}
\Wt(v):=(\wt(v),\swt(v))\in\mathbb R^2. 
\end{equation}
\begin{lemma}
The functions $\wt(*)$, $\swt(*)$, $\Wt(*)$
are additive for products of homogeneous elements of $\mathcal L$.
\end{lemma}

\subsection{$\mathbb Z^2$-gradings}
\begin{lemma} \label{Lhomogeneous}
Consider the Lie algebra $\mathcal L=\Lie(v_1,v_2)\subset \Der R$, the field being arbitrary.
  We have a $\mathbb Z^2$-grading $\mathcal L=\mathop{\oplus}\limits_{a,b\ge 0} \mathcal L_{a,b}$,
  where $\mathcal L_{a,b}$ is spanned by monomials in the generators
  with $a$ factors $v_1$ and $b$ factors $v_2$.
\end{lemma}
\begin{proof}
By construction~\eqref{weights} and~\eqref{Wt},
$\Wt(v_1)=(\lambda,\bar \lambda)$, $\Wt(v_2)=(\lambda^2, \bar\lambda^2)$,
they are linearly independent vectors of $\mathbb R^2$.
Let $w\in \mathcal L_{a,b}$, then by additivity, $\Wt(w)=a\Wt(v_1)+b\Wt(v_2)$.
Consider a pair of integers $(a_1,b_1)\ne (a,b)$ and $w_1\in \mathcal L_{a_1,b_1}$, we have $\Wt(w)\ne \Wt(w_1)$. 
Thus, expressing $w,w_1$ as (probably infinite) linear combinations of monomials (i.e. products of $\dd_i$ and $t_j$)
both elements are expressed via different sets of homogeneous basis elements of $\Der R$.
\end{proof}
In case of a homogeneous element $x\in\mathcal L_{a,b}$ above, define its {\it multidegree} $\Gr(x):=(a,b)\in\Z^2$.
\begin{corollary}\label{Chomogeneous}
Let $\mathcal L=\Lie(v_1,v_2)\subset \Der R$, the field being arbitrary.
  We have a $\mathbb Z$-grading $\mathcal L=\mathop{\oplus}\limits_{n\ge 1} \mathcal L_{n}$,
  where $\mathcal L_{n}$ is spanned by monomials of total degree $n$
  in the generators $\{v_1,v_2\}$.
\end{corollary}

\begin{corollary}\label{Cgrading}
The associative algebra generated by these elements $A:=\Alg(v_1,v_2)\subset \End(R)$, the universal 
enveloping algebra $U(\mathcal L)$, and the Fibonacci (restricted) Lie  algebra $\mathcal L$  (defined below)
are similarly $\mathbb Z^2$ and $\mathbb Z$-graded. 
\end{corollary}

%****************************************************************************************************
\section{Bases of Fibonacci Lie algebra $\mathcal L$ and Fibonacci restricted Lie algebra $\mathbf L$} % ($p=2$)}

From now on we assume that $\ch K=p=2$ and $R=K[t_i| i\ge 0]/(t_i^p| i\ge 0)$, the truncated polynomial ring,
$\dd_i$ its partial derivations, and $R_1\subset R$ the ideal of polynomials without constant term.
We draw attention that now we deal with derivations of the {\it truncated} polynomial ring which is sill  denoted by the same letter $R$.

The goal of this section is to determine standard monomial basis of the appearing Lie algebra,
the related restricted Lie algebra, and show that this Lie algebra does not satisfy a polynomial identity.

\subsection{Fibonacci Lie algebra $\mathcal L$, and its basis}
Define {\it standard monomials} of {\it length} $n:$
\begin{equation}\label{standard}
\begin{split}
W_n& :=\{ t_0^{\alpha_0}\cdots t_{n-4}^{\alpha_{n-4}}v_n \mid \alpha_i\in\{0,1\} \},\qquad n\ge 1;\\
W&:=\mathop{\cup}\limits_{n\ge 1} W_n.
\end{split}
\end{equation}
In particular, $W_i=\{v_i\}$ for $i=1,2,3$.
We write the factor above as $t_0^* \cdots t_{n-4}^*$ for brevity, by
$*$ denoting possible powers $\alpha_i\in\{0,1\}$, this factor we call the {\it tail}. 
For $n=1,2,3$ the tail is equal to 1.
Below, by $r_{n-4}$ we denote also a tail with an additional scalar factor, including zero.
\begin{theorem}
Let $p=2$ and $\mathcal L:=\Lie(v_1,v_2)\subset \Der R$,
the Lie algebra generated by $v_1,v_2$ (i.e. we use the Lie bracket only).
We call $\mathcal L$ the {\bf Fibonacci Lie algebra}.
Then the set $W$~\eqref{standard} is a basis of $\mathcal L$.
\end{theorem}
\begin{proof}
Let us check by induction on $n$ that $W_n$ belong to $\mathcal L$.
As the base of induction we have $W_i=\{v_i\}\subset \mathcal L$ for $i=1,2,3$.
Let $n\ge 3$ and assume that $W_i\subset \mathcal L$ for $i\le n$.
We use the commutation relations of Lemma~\ref{Lcommutators} (by * denote all possible powers below) and obtain
\begin{equation}\label{descrW}
\begin{split}
[v_{n-1},W_n]&=[v_{n-1}, t_0^* \cdots t_{n-4}^* v_n]
  =\{ t_0^* \cdots t_{n-4}^* [v_{n-1},v_n]\} =\{ t_0^*\cdots t_{n-4}^* v_{n+1}\}\subset \mathcal L;\\
[v_{n-2},W_n]&=[v_{n-2}, t_0^* \cdots t_{n-4}^* v_n]
=\{ t_0^*\cdots t_{n-4}^* [v_{n-2},v_n]\} =\{ t_0^*\cdots t_{n-4}^*t_{n-3} v_{n+1} \}\subset \mathcal L.
\end{split}
\end{equation}
These formulas yield the induction step, namely that $W_{n+1}\subset \mathcal L$.

Using Lemma~\ref{Lcommutators} and Lemma~\ref{Laction},
a product of two standard monomials of lengths $n<m$ is expressed via standard monomials:
\begin{equation}\label{prod-stand}
\begin{split}
[r_{n-4}v_n, r'_{m-4}v_m]&=r_{n-4}v_n(r_{m-4}')v_m+r_{n-4}r'_{m-4}[v_n,v_m]\\
&=r_{n-4}v_n(r_{m-4}') v_m+ r_{n-4}r'_{m-4}t_{n-1}\cdots t_{m-3}v_{m+1}\\
&=r''_{m-4}v_m+r''_{m-3}v_{m+1}. 
\end{split} %\qedhere
\end{equation}
\end{proof}

\begin{corollary}\label{CW} We have descriptions of the sets $W_n$:
\begin{enumerate}
\item $W_{n+1}=[v_{n-1},W_n]\cup [v_{n-2},W_n]$,\ $n\ge 3$ (all right hand side monomials are different).
\item
$W_n=\{[v_{n-2,n-3},\ldots,v_{3,2},v_{2,1},v_{3}]\}$, for $n\ge 3$; these are $2^{n-3}$ different elements,
where $v_{i+1,i}$ denotes either $v_{i+1}$ or $v_{i}$ and we consider all such right-normed products.
\item $|W_n|=2^{n-3}$, for $n\ge 3$.
\end{enumerate}
\end{corollary}
\begin{proof} 1) follows from~\eqref{descrW}, 2) follows by induction, and 3) follows from~\eqref{standard}.
\end{proof}

\subsection{Fibonacci restricted Lie algebra $\mathbf L$, and its basis}
Let $p=2$, and $\{v_i\mid i\ge 1\}$ the pivot elements as above,
and $\mathbf L=\Lie_p(v_1,v_2)$ the restricted Lie algebra generated by $v_1,v_2\in \Der R$ with
respect to the $p$th power in $\End R$.
We call $\mathbf L$ the {\bf Fibonacci restricted Lie algebra}.
\begin{lemma} \label{Lsquares}
$v_i^2=t_{i-1}v_{i+2}$ for all $i\ge 1$.
\end{lemma}
\begin{proof} In case $p=2$, an associative algebra satisfies the identity:
\begin{equation}\label{square2}
(a+b)^2=a^2+b^2+[a,b].
\end{equation}
Using~\eqref{vii},
\begin{align*}
v_i^2&=(\dd_i+t_{i-1}v_{i+1})^2=\dd_i^2+t_{i-1}^2v_{i+1}^2+[\dd_i, t_{i-1}v_{i+1}]\\
&=t_{i-1}[\dd_i, v_{i+1}]= t_{i-1}[\dd_i, \dd_{i+1}+t_iv_{i+2}]=t_{i-1}v_{i+2}.
\qedhere
\end{align*}
\end{proof}

We add {\it squares of the pivot elements} given by Lemma~\ref{Lsquares}:
\begin{equation}\label{basisLL}
\begin{split}
\tilde W_n:&=
   \left\{
     \begin{array}{lr}
       W_n, &  n=1,2; \\
       W_n\cup \{t_{n-3}v_n\}, & n\ge 3;
     \end{array}
   \right.\\
\tilde W:&=\mathop{\cup}\limits_{n\ge 1} \tilde W_n.
\end{split}
\end{equation}
%$\tilde W_n:=W_n$ for $n=1,2,3$;
%$\tilde W_n:=$ for $n\ge 3$, and
%$$
\begin{lemma}\label{LbasisLL}
Let $p=2$ and $\mathbf L=\Lie_p(v_1,v_2)\subset \Der R$,
the Fibonacci restricted Lie algebra. 
Then the set $\tilde W$ above~\eqref{basisLL} is a basis of $\mathbf L$.
\end{lemma}
\begin{proof}
By  properties of restricted Lie algebras, in order to obtain the restricted hull $\mathbf L$
of the Lie algebra $\mathcal L$, it is sufficient to add $p^s$-powers for all $s\ge 1$ of elements of its basis $W$~\cite{Ba,StrFar}.
These squares are nontrivial only in case of the pivot elements, the latter given by Lemma~\ref{Lsquares}.
\end{proof}

\subsection{Fibonacci Lie algebra  is not PI}
\begin{theorem}
Let $p=2$ and $\mathcal L=\Lie(v_1,v_2)\subset \Der R$, the Fibonacci Lie algebra.
Then $\mathcal L$ does not satisfy a nontrivial Lie identity.
\end{theorem}
\begin{proof}
By way of contradiction suppose that $\mathcal L$ is PI.
We may suppose that $\mathcal L$ satisfies a nontrivial multilinear polynomial identity as follows~\cite{Ba}  %(commutators are right-normed)
\begin{equation}\label{mon_pi}
\sum_{\pi\in S_n} \alpha_\pi [X_0, X_{\pi(1)},\ldots,X_{\pi(n)}]\equiv 0,\qquad \alpha_\pi\in K,\ \alpha_e=1.
\end{equation}
Substitute $X_0:=v_1$, $X_i:=v_{2i}$ for $i=1,\ldots,n$.
Using Lemma~\ref{Lsosed}, the term for the identity permutation yields %(recall that $p=2$):
\begin{align*}
[v_1,v_2,v_4,v_6,\ldots, v_{2n}]=[v_3,v_4,v_6,\ldots,  v_{2n}] =[v_5,v_6,\ldots, v_{2n}]=\cdots=v_{2n+1}.
\end{align*}
Consider a permutation $\pi$ such that $\pi(1)=k>1$.
By Lemma~\ref{Lcommutators},
the inner product yields $[v_1,v_{2k}]=t_0\cdots t_{2k-3}v_{2k+1}$.
By commutator formula (Lemma~\ref{Lcommutators}), the obtained factor $t_0$ cannot disappear after subsequent products with
the remaining elements $\{v_{2},\ldots,\widehat{v_{2k}},\ldots, v_{2n}\}$.
Hence, products for all such $\pi$ have the factor $t_0$.

Consider permutations with $\pi(1)=1$, the inner product yields $[v_1,v_2]=v_3$, we get monomials
$$
\alpha_\pi[v_3,v_{2\pi(2)},v_{2\pi(3)},\ldots,   v_{2\pi(n)}  ].
$$
Similarly, let $\pi(2)=k> 2$, such products contain
$[v_3,v_{2k}]=t_2\cdots t_{2k-3}v_{2k+1}$ and $t_2$
cannot disappear after subsequent products with the remaining $\{v_4,\ldots,\widehat{v_{2k}},\ldots, v_{2n}\}$.

Now consider a general nonidentity permutation $\pi$. There exists $s\in \{1,\ldots, n-1\}$ such that
$\pi(i)=i$ for $i=1,\dots,s-1$ but $\pi(s)=q>s$. Similarly, consider respective monomial in~\eqref{mon_pi}
\begin{align*}
[v_1,v_2,v_4,\ldots,v_{2(s-2)},v_{2(s-1)},v_{2q}, v_{2\pi(s+1)},\ldots,v_{2\pi(n)} ]\\
=[v_{2s-1},v_{2q}, v_{2\pi(s+1)},\ldots,v_{2\pi(n)} ]\\
=[t_{2s-2}\cdots t_{2q-3}v_{2q+1},v_{2\pi(s+1)},\ldots,v_{2\pi(n)}  ]
\end{align*}
Since $q>s$ we get the factor $t_{2s-2}$ that cannot disappear by further products with\\ $\{v_{2s},\ldots, \widehat{v_{2q}},\ldots, v_{2n} \}$.
We get a contradiction because the term for the identity permutation cannot cancel with anything.
\end{proof}

\subsection{Estimates on the weight function}
The following result shows that the weight function $\wt(*)$ separates and stratifies the basis sets $W_n$ on plane (see also below).
The respective borders of these sets are shown as the red lines on Fig.~\ref{Fig1}, Fig~\ref{Fig2}, and Fig~\ref{Fig3}.
\begin{lemma}\label{LwtW}
$\lambda^{n-1}<\lambda^{n-1}+\lambda \le \wt(W_n)\le \lambda^n$ for $n\ge 1$.
\end{lemma}
\begin{proof}
By~\eqref{standard}, since $\wt(t_i)<0$, the upper bound is trivial.
Since $\wt(t_i^*)\ge \wt(t_i)=-\lambda^i$, and $\lambda^2-\lambda=1$, we get
\begin{align*}
&\wt(t_0^*\cdots t_{n-4}^*v_n)\ge \wt(v_n)-1-\lambda-\cdots- \lambda^{n-4}
=\lambda^n-\frac{\lambda^{n-3}-1}{\lambda-1}\\
&\quad =\lambda^n-\frac{\lambda^{n-2}-\lambda}{\lambda^2-\lambda}=\lambda^n-\lambda^{n-2}+\lambda=
\lambda^{n-2}(\lambda^2-1)+\lambda
=\lambda^{n-1}+\lambda>\lambda^{n-1}.
\qedhere
\end{align*}
\end{proof}
\begin{corollary} Both estimates above are exact.
\end{corollary}

%**********************************************************************************************************************
\section{Nillity of the Fibonacci restricted Lie algebra $\mathbf L$}
\subsection{Nillity}
The following fact is an analogue of the periodicity of the Grigorchuk and Gupta-Sidki groups.
Here we present the original and a new proof, that yields new bound on nillity indices of elements.
We say that a restricted Lie algebra $L$ has a {\it nil $p$-mapping} provided that for any $x\in L$ there exists a number $n(x)\in\N$
such that $x^{[p^{n(x)}]}=0$. 
\begin{theorem}
Let $\ch K=2$, and consider the Fibonacci restricted Lie algebra $\mathbf L=\Lie_p(v_1,v_2)$.
Then   $\mathbf L$ has a nil $p$-mapping.
\end{theorem}
\begin{proof}
{\bf New proof} (bounds from below).
By Lemma~\ref{LbasisLL}, $\mathbf L$ is spanned by standard monomials and squares of pivot elements.
We write any $a\in\mathbf L$ as
\begin{equation}\label{summochka2}
a=\sum_{j=n}^{m} r_{j-3} v_j,\qquad   r_j\in R.% \ r_{n-3}\in R_1.
\end{equation}
%(In case the expansion starts with $a=v_1+v_2+\cdots $, we consider $a^2$, and proceed as below.)
Denote  $a_1:=r_{n-3} v_{n}+ r_{n-2} v_{n+1}$, $a_2:=\sum_{j=n+2}^{m} r_{j-3} v_j$.
So, $a=a_1+a_2$.
%By assumption~\eqref{summochka2}, we have $r^2_{n-3}=0$.
Using~\eqref{square2}, Lemma~\ref{Lcommutators}, Lemma~\ref{Lsquares}, and computations~\eqref{prod-stand}, we get 
\begin{align*}
a^2&=(r_{n-3} v_{n}+ r_{n-2} v_{n+1}+a_2)^2\\
&=r^2_{n-3} v^2_{n}+ r^2_{n-2} v^2_{n+1}+r_{n-3}r_{n-2}[v_n,v_{n+1}]+ [a_1,a_2]+a_2^2\\
&=r^2_{n-3}t_{n-1}v_{n+2}+ r^2_{n-2}t_{n}v_{n+3}+  r_{n-3}r_{n-2}v_{n+2}+ \sum_{i=n+3}^{m'}\tilde r_{i-3} v_i\\
&=\sum_{i=n+2}^{m'} r'_{i-3} v_i,\quad r'_i\in R,
\end{align*}
indeed, the extreme case in computing of $[a_1,a_2]$ above is
\begin{align*}
[r_{n-3}v_n, r_{n-1}v_{n+2}]&=r_{n-3}v_n (r_{n-1})v_{n+2}+ r_{n-3}r_{n-1} [v_n , v_{n+2}]\\
&=r'_{n-1}t_{n-1}v_{n+3}=r''_{n-1}v_{n+3}. 
\end{align*}

Thus, we have a presentation like~\eqref{summochka2}, where the lower index shifted by 2.
By iteration, we get
\begin{align} \label{summochka3}
a^{2^N}=\sum_{i=n+2N}^{m''} r''_{i-3} v_i, \quad r''_i\in  R,\qquad N\ge 0.
\end{align}
Assume that $a^{2^N}\ne 0 $. By the lower bound of Lemma~\ref{LwtW},
\begin{equation*}
\wt (a^{2^N})\ge \min_{i\ge n+2N}\wt ( r''_{i-3} v_i)> \lambda^{n+2N-1}.
\end{equation*}
On the other hand, by property of the weight function and~\eqref{summochka2}
\begin{equation*}
\wt (a^{2^N})\le 2^N\wt a\le 2^N \wt (r_{m-3} v_{m})\le 2^N \lambda^m.
\end{equation*}
These two bounds yield
\begin{equation}\label{est_low}
N< C( m-n+1), \qquad C:=\ln_{\lambda^2/2} \lambda\approx 1,787.
\end{equation}
{\bf Old proof}~\cite{Pe06} (bounds from above). 
By Lemma~\ref{LbasisLL}, $\mathbf L$ is spanned by standard monomials and squares of pivot elements.
So, we write any $a\in\mathbf L$ roughly as
\begin{equation}\label{summochka}
a=\sum_{n=1}^{s-1} r_{n-3} v_n +h_{s-3}v_s,\qquad r_i\in R,\ h_{s-3}\in R_1,
\end{equation}
where we assume that the senior tail has no scalar terms
(i.e. $h_{s-3}\in R_1$, it may be zero).
Compute $a^2$ using~\eqref{square2}.
Mutual commutators of terms~\eqref{summochka} are computed similar to~\eqref{prod-stand} using Lemma~\ref{Lcommutators} and Lemma~\ref{Laction}:
\begin{align*}
[r_{n-3} v_n, r_{m-3} v_m]
  &=r_{n-3}v_n(r_{m-3}) v_m+ r_{n-3}r_{m-3} [v_n,v_{m}]\\
  &=r'_{m-3} v_m + r''_{m-3} t_{n-1}\cdots t_{m-3}v_{m+1} \\
  &=r'_{m-3}v_m+ r'''_{m-3} v_{m+1}, \qquad\qquad\qquad n<m< s;\\
[r_{n-3} v_n, h_{s-3} v_s]
 &=r_{n-3}v_n(h_{s-3})v_s + r_{n-3}h_{s-3}[v_n,v_{s}]\\
 &=r'_{s-3}v_s+ h'_{s-3}t_{n-1}\cdots t_{s-3}v_{s+1} \\
 &=r'_{s-3}v_s+ h''_{s-3}v_{s+1},\qquad h''_{s-3}\in R_1, \qquad\quad n<s,
\end{align*}
where $r_*^*\in R$.  Squares of terms~\eqref{summochka} are
\begin{align*}
(r_{n-3} v_n)^2 &= r_{n-3}^2 t_{n-1}v_{n+2}= h'_{n-1}v_{n+2},\qquad h'_{n-1}\in R_1,\qquad \ n<s;\\
(h_{s-3} v_s)^2&=h_{s-3}^2 v_s^2 =0,
\end{align*}
because $h_{s-3}\in R_1$.
In case $n=s-1$, the first line above yields $h'_{s-2}v_{s+1}$.
Combining all these cases, we get a presentation of the same form as~\eqref{summochka}
\begin{equation*}
a^2=\sum_{n=1}^{s} \bar r_{n-3} v_n +\bar h_{s-2}v_{s+1},\qquad \bar r_i\in R,\ \bar h_{s-2}\in R_1.
\end{equation*}
Repeating the process, we get
\begin{equation}\label{summochka4}
a^{2^m}\!\!=\!\!\sum_{n=1}^{s+m-1}\!\! \tilde r_{n-3} v_n +\tilde h_{s+m-3}v_{s+m},\quad
\tilde r_i \in R,\ \tilde h_{s+m-3} \in R_1, \quad m\ge 1.
\end{equation}
Assume that $a^{2^m}\ne 0$.
Since $\wt(\tilde h_{s+m-3})<0$, by Lemma~\ref{LwtW}, we obtain $\wt(a^{2^m})< \wt (v_{s+m})= \lambda^{s+m} $,  
the inequality denoting the bound on weights of all homogeneous summands.
Weights of the generators of the algebra $\mathbf L$ are $\wt(v_1)=1$, $\wt(v_2)=\lambda>1$.
By additivity of the weight function, $\wt (a)\ge 1$ and $\wt (a^{2^m})\ge 2^m$, the inequality being valid for all homogeneous summands.
We obtain $2^m< \lambda^{m+s}$.  Hence
\begin{equation}\label{est_up}
m< C_1 s,\qquad C_1:=\log_{2/\lambda } (\lambda)\approx 2,27.\qedhere   
\end{equation}
% where $\lambda\approx 1.618$, which is false for sufficiently large $m$.
% The contradiction proves that $a^{2^m}=0$ for sufficiently large $m$ (depending on $a$).
\end{proof}

\subsection{Estimate on Nilpotency indices}
Two different proofs above combined yield an estimate better than both estimates~\eqref{est_low} and \eqref{est_up}  alone.
\begin{corollary}
Let $0\ne a\in\mathbf L$ be written via its basis as
$ a=\sum\limits_{j=n}^{m} r_{j-3} v_j$,  $r_j\in R$. Then
$$a^{2^{m-n+2}}=0.$$
\end{corollary}
\begin{proof}
Putting $h_{m-2}=0$, we write $a$ as~\eqref{summochka}
\begin{equation}
a=\sum_{j=n}^{m} r_{j-3} v_j+ h_{m-2}v_{m+1},\qquad h_{m-2}\in R_1.
\end{equation}
We compute combining~\eqref{summochka3} and~\eqref{summochka4}
$$
a^{2^N}=\sum_{i=n+2N}^{m+N} r''_{i-3} v_i + h''_{m+N-2}v_{m+N+1}, \quad r''_i\in  R,\quad h''_{m+N-2}\in R_1,\qquad N\ge 0.
$$
If $N=m-n+1$, then the sum above is trivial, and we have the last term only. Its square equals zero.
\end{proof}

\begin{conjecture}
This bound on nilpotency indices is optimal, provided by elements
$$a=v_{n}+v_{n+1}+\cdots+ v_{m}.$$
\end{conjecture}

\begin{Remark}
Another approach to prove nillity was suggested by I.~Shestakov and E.~Zelmanov~\cite{ShZe08}.
This approach was further developed and applied to another examples in~\cite{Bartholdi10,Pe17,Pe20clover,Pe20flies}.
\end{Remark}

%\section{Growth and nillity of $\mathbf L$}
\section{Growth} %of Fibonacci (restricted) Lie algebra}

In this section, we show that the Fibonacci Lie algebra  and its restricted hull have a slow polynomial growth.
Moreover, we describe the growth in more details than in~\cite{PeSh09,PeSh13fib}.
We show that the growth is somewhat uneven, which can be also observed on Fig.~\ref{Fig2} and Fig.~\ref{Fig3}.
We also specify the intermediate growth of respective enveloping algebras despite the uneven growth of the Lie algebra $\mathcal L$.
\subsection{Growth of $\mathcal L$ and $\mathbf L$}
\begin{theorem} \label{TgrowthP}
Let $\ch K=2$, $\mathcal L=\Lie(v_1,v_2)$, $\mathbf L=\Lie_p(v_1,v_2)$,  and $\lambda=(1+\sqrt 5)/2$.
These algebras have the following Gelfand-Kirillov dimensions (coinciding with $\Dim^2\mathbf L$, $\LDim^2 \mathbf L$, see Lemma~\ref{L1}):
$$\GKdim \mathbf L=\LGKdim \mathbf L= \GKdim \mathcal L=\LGKdim \mathcal L=\log_\lambda  2\approx 1.44042.$$
\end{theorem}
\begin{proof}
Consider a real number $x\ge 1$.
In order to specify the growth we use a {\it weight growth function}
$\tilde \gamma_{\mathcal L}(x):=\dim_K \langle w\mid w\in W,\ \wt(w)\le x\rangle_K$, $x\ge 0$.
Let $w\in W$ be of length $n$ of weight at most $x$.
Using Lemma~\ref{LwtW}, we get $\min \{x,\lambda^n \}\ge \wt(w)> \lambda^{n-1}$.
Hence,  $n\le n_0:=\lceil\log_\lambda x\rceil$.
By Corollary~\eqref{CW}, $|W_n|=2^{n-3}$ for all $n\ge 3$.
We evaluate
$\tilde \gamma_{\mathcal L}(x)$ by the number of standard monomials of length at most $n_0< 1+\log_\lambda x$
\begin{align*}
\tilde \gamma_{\mathcal L}(x)&\le |W_1\cup\cdots\cup W_{n_0}|=1+1+1+2+\cdots +2^{n_0-3}
=1+2^{n_0-2}\\
&< 1+2^{\log_\lambda x-1}=1+\frac 12 x^{\log_\lambda 2}.
\end{align*}
Using the same $x,n_0$ and  Lemma~\ref{LwtW},
$\wt (W_{n_0-1})\le \lambda^{n_0-1}=\lambda^{\lceil \log_\lambda x\rceil -1} < \lambda^{\log_\lambda x} = x$.
We get a lower bound
\begin{align*}
\tilde \gamma_{\mathcal L}(x)&\ge  |W_1\cup\cdots\cup W_{n_0-1}|=1+1+1+2+\cdots +2^{n_0-4}\\
&>2^{n_0-3} \ge 2^{\log_\lambda x-3}=\frac 18 x^{\log_\lambda 2}.
\end{align*}
These bounds yield the required Gelfand-Kirillov dimensions of $\mathcal L$.

In case of the restricted Lie algebra $\mathbf L$, consider a square $u:=t_{n-3}v_n=v_{n-2}^2$ of weight at most $x$.
Then $x\ge \wt(u)=2\wt(v_{n-2})=2\lambda^{n-2}$.
So, $n\le n_1(x):=2+ \lceil\log_\lambda (x/2)\rceil$.
We obtain bounds
$\tilde \gamma_{\mathcal L}(x) \le \tilde \gamma_{\mathbf L}(x)\le \tilde \gamma_{\mathcal L}(x)+n_1(x)$,
yielding the same upper and lower Gelfand-Kirillov dimensions.
\end{proof}

\begin{corollary} \label{Cbounds}
Computations above yield the following polynomial bounds for $\mathcal L$
\begin{enumerate}
\item $\displaystyle \frac 18 n^{\log_\lambda 2}\le  \tilde \gamma_{\mathcal L}(n)\le 1+\frac 12 n^{\log_\lambda 2} $, for $n\ge 1$ (weight growth function);
\item $\displaystyle\frac {\lambda^{\log_\lambda 2}}8 n^{\log_\lambda 2}\le  \gamma_{\mathcal L}(n)\le 1+\frac {\lambda^{2\log_\lambda 2}  }2 n^{\log_\lambda 2} $ for $n\ge 1$ (regular growth function).
\end{enumerate}
\end{corollary}
\begin{proof}
The bounds of the first claim were established  in Theorem~\ref{TgrowthP}.
The second claim follows from  bounds
$\tilde \gamma_{\mathcal L} (\lambda n)\le \gamma_{\mathcal L} (n)\le \tilde\gamma_{\mathcal L} (\lambda^2 n) $ for all $n\ge 1$.
To check this bound, by assigning weights one to the generators $v_1,v_2$, we get the regular growth function.
If the generators have weights $(\lambda,\lambda^2)$ then we get the weight growth function.
Replacing weights by either $(\lambda,\lambda)$ or $(\lambda^2,\lambda^2)$ we get the bounds.
\end{proof}

\subsection{Nonuniform polynomial growth of $\mathcal L$}
Let us show that the growth function $\gamma_{\mathcal L }(n)$ behaves unevenly.
Namely, introduce the function that counts monomials of length exactly $n$, namely set
$s_{\mathcal L }(n):=\gamma_{\mathcal L }(n)-\gamma_{\mathcal L }(n-1)$ (before we used notation $\lambda_{\mathcal L}(n)$ in~\cite{PeSh09,PeSh13fib}).
The following fact can be observed by counting monomials on two parallel lines $x+y=\mathrm{const}$ in Fig.~\ref{Fig1},
the first passing through $v_n$, the second one step up (they are not parallel to red lines).

\begin{lemma}\cite[Corollary 6.5]{PeSh13fib}\label{Lnonuniform}
$s_{\mathcal L }(F_n)=s_{\mathcal L }(F_n+1)= 2$ for  $n \ge  4$, where $F_n$ is the $n$th Fibonacci number.
\end{lemma}

The nonuniform polynomial behaviour is shown also by the following.
Namely, there is no unique constant instead of $1/2$ and $1/8$ in Claim 1 of Corollary~\ref{Cbounds}.
\begin{theorem}\label{Lnonuniform2}
Does not exist the limit for the weight growth function
$$\displaystyle \not\!\exists \lim\limits_{n\to \infty }\frac {\tilde\gamma_{\mathcal L}(n)}{n^{\log_\lambda 2}}.$$
\end{theorem}
\begin{proof} Set $x=x(n):=\lambda^n $ for some $n\in \N$. Then by Lemma~\ref{LwtW} and Corollary~\ref{CW}
\begin{align}\nonumber
\tilde \gamma_{\mathcal L}(x)&= |W_1\cup\cdots\cup W_{n}|=1+1+1+2+\cdots +2^{n-3}\\
&=1+2^{n-2}= 1+2^{\log_\lambda x-2}=1+\frac 14 x^{\log_\lambda 2}.
\label{e1}
\end{align}

On the other hand, consider $y=y(n):=\lambda^n+\lambda^{n-2}<\lambda^n+\lambda^{n-1}=\lambda^{n+1}$, $n\in\N$.
By Lemma~\ref{LwtW}, the weight growth function at $y$
counts all monomials $W_1\cup\cdots\cup W_n$ and some monomials of the set $W_{n+1}=[v_{n-1},W_n]\cup [v_{n-2},W_n]$ (see Corollary~\ref{CW}).
(On Fig.~\ref{Fig1} the border line for such monomials is the dashed blue line shown and dividing the set $W_9$.)
Since $\wt([v_{n-2},W_n])\le \lambda^{n-2}+\lambda^n=y$ we count all monomials $[v_{n-2},W_n]$.
By Corollary~\ref{CW}, there are $|[v_{n-2},W_n]|=|W_n|=2^{n-3}$ such monomials.
Also, by~\eqref{descrW} we get monomials $W_{n+1}':=\{t_0^*\cdots t_{n-6}^*t_{n-5}t_{n-4} v_{n+1}\}\subset [v_{n-1},W_n]$.
These monomial are as required because $\wt(W_{n+1}')\le \lambda^{n+1}-\lambda^{n-4}-\lambda^{n-5}=\lambda^{n}+\lambda^{n-1}-\lambda^{n-3}=\lambda^{n}+\lambda^{n-2}=y$.
We obtain $|W_{n+1}'|=2^{n-5}$ monomials. Thus we have estimates
\begin{align*}
\tilde \gamma_{\mathcal L}(y)&\ge |W_1\cup\cdots\cup W_{n}|+|[v_{n-2},W_n]|+|W_{n+1}'| \\
&=1+1+1+2+\cdots +2^{n-3}+2^{n-3}+2^{n-5}\\
&=1+2^{n-2}+2^{n-3}+2^{n-5}>13\cdot 2^{n-5}.
\end{align*}
Since $y=(\lambda^2+1)\lambda^{n-2}$, we get $n=2+\log_\lambda y-\log_{\lambda}(\lambda^2+1)$, this value we substitute above
\begin{equation}\label{e2}
\tilde \gamma_{\mathcal L}(y)> 13\cdot 2^{-3 +\log_\lambda y-\log_{\lambda}(\lambda^2+1)}
=\frac {C} 4 y^{\log_{\lambda} 2},\quad C:= \frac {13}{2^{1+\log_\lambda(\lambda^2+1)}}\approx 1,0197.
\end{equation}
Two sequences $x(n)$ and $y(n)$, $n\in\N$, and estimates~\eqref{e1},~\eqref{e2} yield the claim.
\end{proof}
\begin{Remark}
On Fig.~\ref{Fig2} and Fig~\ref{Fig3} one can observe that monomials of $\mathcal L$ are more concentrated in the middle of the sets $W_n$, 
while vicinities of the pivot elements $v_m$  are less occupied, which is shown by Lemma~\ref{Lnonuniform} as well.
So, this nonuniform distribution of monomials influences the function $\tilde\gamma_{\mathcal L}(n)$ counting the total number of all monomials till the weight $n$,
see also Theorem~\ref{Lnonuniform2}.  
\end{Remark}

More computations and geometric observations are needed to establish an analogous fact.
\begin{conjecture}
Does not exist the limit for the regular growth function $\not\!\exists \lim\limits_{n\to \infty } \gamma_{\mathcal L}(n)/n^{\log_\lambda 2}$ as well.
\end{conjecture}

\subsection{Intermediate Growth of restricted and universal enveloping algebras}
By Theorem~\ref{TgrowthP},  $\Dim^2\mathbf L=\GKdim{\mathbf L}=\lambda$.
We can describe the intermediate growth
of the universal enveloping algebra $U(\mathbf L)$ using  \cite[Proposition 1]{Pe96} as follows
\begin{equation}\label{grUupper}
\lambda-1\le \Dim^3 U({\mathbf L})\le \lambda.
\end{equation}
The upper bound~\eqref{grUupper} is equivalently rewritten as
\begin{equation}\label{boundU}
\begin{split}
\gamma_U(n)&\le \Phi^3_{\alpha+o(1)}(n)=\exp(n^{\theta+o(1)}),\quad  n\to\infty,\qquad \text{where}\\
\theta&:=\frac{\lambda}{\lambda+1}= \frac {\log_\lambda 2}{\log_\lambda 2+1}=\log_{2\lambda}2=\log_{\sqrt 5+1}2\approx 0.5902.
\end{split}
\end{equation}

The lack of a precise  value is caused by the nonuniform behaviour of the growth function $s_{\mathcal L} (n)$ (see Lemma~\eqref{Lnonuniform})
and we cannot apply~\cite{Pe96} to specify the growth of universal algebras explicitly.

Now we explicitly determine the growth of universal enveloping algebras despite the uneven growth of the Lie algebra $\mathcal L$.
\begin{theorem} \label{TgrowthU}
Let $\ch K=2$, $\mathcal L=\Lie(v_1,v_2)$, $\mathbf L=\Lie_p(v_1,v_2)$, and $\lambda=(1+\sqrt 5)/2$.
Let $U$ denote any of three algebras:
universal enveloping algebras $U(\mathcal L)$, $U(\mathbf L)$,
or the restricted enveloping algebra $u(\mathbf L)$.
Then
$$\LDim^3 U =\Dim^3 U =\log_\lambda  2\approx 1.44042.$$
\end{theorem}
\begin{proof}
We have the upper bound by~\eqref{grUupper}.
Since the PBW-monomials of $U(\mathcal L)$ and $u(\mathbf L)$ form a subset of standard monomials of $U(\mathbf L)$, 
we have the upper bound on the growth of all three enveloping algebras.

Let us prove the lower bound on growth of any of our three enveloping algebras.
Fix $x> 0$. Consider $n\in \N$.
Denote the set of elements~\eqref{standard} as $T_n=:\{w_1,\dots, w_N\}$, where $N:=|T_n|=2^{n-3}$ by Corollary~\ref{CW}.
We take all ordered products $w=w_1^*w_2^*\cdots w_N^*$, for all $*\in\{0,1\}$.
Then $\wt w\le \lambda^n|T_n|=\lambda^n 2^{n-3}$.
Now we want their weights to be bounded, namely $\lambda^n 2^{n-3}\le x$.
To this end it is sufficient to set $n_0:=[\log_{2\lambda}(8x)]$.
These ordered products belong to the PBW-bases for our enveloping algebras. Then, the following bound proves the required lower estimate (we use~\eqref{boundU})
\begin{align*}
\tilde \gamma_U(x)&> 2^{|T_{n_0}|}=\exp(\ln 2\cdot 2^{n_0-3})=\exp (\ln 2 \cdot 2^{[\log_{2\lambda}(8x)]-3})\\
&\ge \exp (\ln 2 \cdot 2^{\log_{2\lambda}(8x)-4}) 
=\exp \big(2^{-4}\ln 2 \cdot (8x)^{\log_{2\lambda}2}\big)\\
&=\exp \Big(2^{-4}8^{\log_{2\lambda}2}\ln 2  \cdot x^{\theta}\Big),\qquad\text{where}\quad \theta=\log_{2\lambda}2=\frac{\lambda}{\lambda+1}.
\qedhere
\end{align*}
\end{proof}
\begin{corollary} \label{Cinter}
In notations of~\eqref{scale1} and Lemma~\ref{L1}, the statement of Theorem is equivalently written as
$$
\gamma_U(n)=\exp\big( n^{\theta+o(1)} \big ), \quad n\to \infty; \qquad \text{where}\quad
\theta=\frac{\lambda}{\lambda+1}=\log_{\sqrt 5+1}2 \approx 0.5902.
$$
\end{corollary}

%***************************************************************************
\section{Geometric ideas}%%: monomials of $\mathcal L$ on plane belong to a strip}
\label{Sgeometry}

In this section we develop geometric ideas and their applications introduced  in~\cite{PeSh09,PeSh13fib}.
\subsection{Estimates on superweight function}
The next result yields the strip on Fig~\ref{Fig1}, see below. 
\begin{lemma} \label{Leta}
Consider $w\in \tilde W$, i.e. a basis monomial of $\mathbf L=\Lie_p(v_1,v_2)$. Then
$$ -\lambda <\swt w< 1. $$
\end{lemma}
\begin{proof}
Consider a standard monomial~\eqref{standard} $w=t_0^*\cdots t_{n-4}^*v_n$, $n\ge 0$.
Using \eqref{weights} and \eqref{lambd}, we have
$\swt t_i=-\swt v_i=-{\bar \lambda}^{i}=-(-1/\lambda)^i$, $i\ge 0$.
The lower bound is given by monomials having $t_i$ with even indices only
\begin{equation*}
\swt w\ge-\frac 1 {\lambda^{n}}+\! \sum _{i= 0}^ {2i\le n-4} \frac {-1}{\lambda^{2i}}
> \sum_{i\ge 0} \frac {-1}{\lambda^{2i}}= \frac {-1}{1-1/\lambda^2}=\frac {-\lambda^2}{\lambda^2-1}=\frac {-\lambda^2}\lambda=-\lambda.
\end{equation*}
The upper bound is given by monomials having $t_i$ with odd indices only
\begin{equation*}
\swt w\le \frac 1{\lambda^{n}} +\!\!  \sum _{i= 0}^ {2i+1\le n-4} \frac 1{\lambda^{2i+1}}
< \sum_{i\ge 0} \frac {\lambda^{-1}}{\lambda^{2i}}= \frac {\lambda^{-1}}{1-1/\lambda^2}
=\frac {\lambda}{\lambda^2-1}=\frac {\lambda}\lambda=1.
\end{equation*}
By Lemma~\ref{LbasisLL}, it remains to consider squares of the pivot elements $w=v_i^2=t_{i-1}v_{i+2}$, $i\ge 1$.
Then $\swt w=\swt v_i^2=2\bar\lambda^i   =2 (-1/\lambda)^i$.
It is sufficient to see that $\swt v_1^2=-2/\lambda=-2(\lambda-1)\approx -1,236>-\lambda$
and $\swt v_2^2=2/\lambda^2<1$.
\end{proof}
\begin{corollary}
There are basis elements of $\mathcal L=\Lie(v_1,v_2)$ with superweights arbitrarily close to $-\lambda$ and $1$.
\end{corollary}
\begin{proof}
By computations above,  the standard monomials having $t_i$s with all odd or all even indices  are as required.
\end{proof}

\subsection{Weight coordinates on plane}
By Lemma~\ref{Lhomogeneous} we have a $\Z^2$-grading.
Consider $w\in {\mathcal L}_{a,b}$, where $(a,b)\in \N_0^2$,
i.e. $w$ is a linear combination of commutators having $a$ factors $v_1$ and $b$ factors $v_2$.
By additivity of the weight functions, $\Wt w= a\Wt v_1+b\Wt v_2$.
By~\eqref{Wt} and~\eqref{lambd}, the components yield
\begin{align*}
\wt w&=a\wt v_1+b\wt v_2\quad =a\lambda+b\lambda^2;\\
\swt w&=a\swt v_1+b\swt v_2=a\bar \lambda+b\bar \lambda^2=-a/\lambda+b/\lambda^2.
\end{align*}
This observation motivates us  to introduce the {\it weight coordinates} $(\xi,\eta)$ on plane $\R^2$:
\begin{equation}\label{newcoord}
\left\{
\begin{split}
\xi  &:=x\lambda+y\lambda^2;\\
\eta &:=-x/\lambda+y/\lambda^2;
\end{split}\right.
\qquad\quad (x,y)\in \R^2.
\end{equation}
Let $w\in {\mathcal L}_{a,b}$ be a monomial, which we put on plane using its integer multidegree coordinates $(x,y):=\Gr(w)=(a,b)\in\Z^2\subset \R^2$.
By our construction, we have also the weight coordinates $(\xi,\eta)=\Wt w=(\wt w,\swt w)$.
\begin{lemma}
These axis $\xi,\eta$ on plane $\R^2$ are orthogonal.
\end{lemma}
\begin{proof}
The axis $\xi$ is given by points with $\eta=0$, i.e. $y=\lambda x$.
The axis $\eta$ is given by $\xi=0$, i.e. $y=-x/\lambda$.
These lines are orthogonal.
\end{proof}

\subsection{Strip for Fibonacci Lie algebras $\mathcal L$ and $\mathbf L$}
\begin{theorem}\label{Tstrip}
Let $p=2$ and consider the Fibonacci restricted Lie algebra $\mathbf L=\Lie_p(v_1,v_2)$.
Put its basis monomials $w\in \mathbf L_{a,b}$ on plane using
the multidegree coordinates $(x,y):=\Gr(w)=(a,b)\in\Z^2\subset \R^2$. These points lie in the following strip:
$$\lambda x -\lambda^3<  y< \lambda x +\lambda^2 .$$
\end{theorem}
\begin{proof}
By Lemma~\ref{Leta} and~\eqref{newcoord} we have
$ -\lambda<\eta=\swt w=-x/\lambda+y/\lambda^2<1$.
\end{proof}
\begin{corollary} \label{CstWn}
The standard monomials $W_n$  of fixed length $n$, $n\ge 1$,  belong to different rectangles that
compose the strip above. Namely, in terms of the weight coordinates we have (see Fig.~\ref{Fig1}, Fig.~\ref{Fig2}, and Fig.~\ref{Fig3})
\begin{equation*}
\left\{
\begin{split}
\lambda^{n-1} &< \xi=\wt w\le \lambda^n,\\
-\lambda &<  \eta= \swt w< 1,
\end{split}\right.
\qquad w\in W_n,\ n\ge 1.
\end{equation*}
\end{corollary}
\begin{proof}
Alternatively, the inequalities follow from Lemma~\ref{LwtW} and Lemma~\ref{Leta} directly.
\end{proof}

\begin{theorem}\label{Tsum2}
The Fibonacci (restricted) Lie algebra is a direct sum of two locally nilpotent subalgebras:
$$ \mathcal L=\mathcal L_- \oplus \mathcal L_+,\qquad \mathbf L=\mathbf L_- \oplus \mathbf L_+. $$
\end{theorem}
\begin{proof}
We use construction of the $\Z^2$-grading in Lemma~\ref{Lhomogeneous}.
Let $0\ne x\in \mathbf L_{(a,b)}$, $a,b\in \Z$; then $\swt x=a\swt v_1+b\swt v_2=   a\bar \lambda +b\bar\lambda^2\ne 0$ because $\bar\lambda $ is irrational.
Let $\mathbf L_+$, $\mathbf L_-$ be spanned by all basis monomials of positive and negative superweights, respectively.
We get a decomposition into a direct sum of subspaces that are subalgebras because the superweight function is additive.

The nillity of subalgebras easily follows by geometric arguments. Indeed, consider $L:=\Lie_p(u_1,\dots,u_k)$, where $u_i\in \mathbf L_+$.
Let $\mu:=\min \{\swt u_i\mid i=1,\ldots, k\}$.
We have $\mu>0$, set $N:=\lceil1/\mu\rceil$. 
By the additivity of the superweight function, $\swt (L^N)\ge 1$.
Using the upper estimate of Lemma~\ref{Leta} we conclude that $L^{N}=0$.
\end{proof}
\begin{corollary}
Similar decompositions into direct sums of two subalgebras are also valid in the setting of an arbitrary characteristic
and for associative enveloping algebras, universal, or restricted enveloping algebras. 
\end{corollary}

\subsection{On structure of Fibonacci Lie algebra $\mathcal L$}
We naturally extend the shift $\tau$ (see subsection~\ref{SSpivot}) on the standard monomials $W$.
Denote $W_{\le n}:=\cup_{i=1}^n W_i$ for all $n\ge 1$.
\begin{lemma}\label{LWrecc}
We have recursive relations for the standard monomials in terms of disjoint unions
\begin{align}
W_{\le n+1}&=\{v_1\}\bigcup \tau(W_{\le n}) \bigcup t_0\tau(W_{\le n})\setminus \{t_0v_2, t_0v_3\},\qquad n\ge 2;\nonumber\\
W&=\{v_1\}\bigcup \tau(W) \bigcup t_0\tau(W)\setminus \{t_0v_2, t_0v_3\}\label{decomp}  \\
&=\{v_1\}\bigcup \tau(W) \bigcup t_0\tau(W  {\setminus} \{v_1, v_2\}).\nonumber
\end{align}
\end{lemma}
\begin{proof} Follows from the structure of the standard monomials~\eqref{standard}.
\end{proof}

Let $A\subset \mathcal L$ be linear span of standard monomials~\eqref{standard} $W$ that contain $t_0$.
\begin{theorem}\label{T_L_decomp}
Let $p=2$, consider the Fibonacci Lie algebra $\mathcal L=\Lie(v_1,v_2)$.
\begin{enumerate}
\item
We have a semidirect product decomposition
\begin{equation}\label{abelian}
\mathcal L=\langle v_1\rangle _K \rightthreetimes ( L_{(2)} \rightthreetimes A),
\end{equation}
where $L_{(2)}=\Lie(v_2,v_3)=\tau(\mathcal L)\cong \mathcal L$ is a subalgebra,
$A:=\big \langle t_0 \tau(W)\setminus \{t_0v_2, t_0v_3\}\big \rangle _K$ an abelian ideal.
\item $\mathcal L/A\cong \Lie(\dd_1,v_2)\cong \langle v_1\rangle _K \rightthreetimes L_{(2)}$.
\end{enumerate}
\end{theorem}
\begin{proof}
Factors $t_0$ cannot disappear via commutators because all pivot elements do not contain $\dd_0$.
Hence, $A$ is an abelian ideal in $\mathcal L$.
The semidirect product decomposition follows from~\eqref{decomp} and basic relations (subsection~\ref{SSbasic}).
Recall that, $v_1=\dd_1+t_0v_2$, hence commutators of $v_1$ and $\dd_1$ with the basis elements $W$ coincide modulo $A$.
Observe also that $[\dd_1,v_2]=[\dd_1,\dd_2+t_1v_3]=v_3$, so $\Lie(\dd_1,v_2)\cong \langle v_1\rangle _K \rightthreetimes L_{(2)}$.
Thus, we get the second claim.
\end{proof}

Let us show that monomials of the abelian ideal $A$ belong to the lower part of the strip.
\begin{lemma}
Monomials for the abelian ideal $A\triangleleft\mathcal L$  are between the $\xi$-axis and the lower line of the strip (Fig.~\ref{Fig1}).
Namely, consider a standard basis element $w\in W\cap A$, then
$$-\lambda <\swt w  <0. $$
\end{lemma}
\begin{proof}
By construction of  $A$, $w=t_0t_1^*\cdots t_{n-4}^* v_n\in A$, $n\ge 4$.
The lower bound is trivial.
Recall that $\swt(t_i)=-\swt(v_i)=-(-1/\lambda)^i$, $i\ge 0$.
We get the upper bound considering variables with odd indices, and using~\eqref{lambd}
\begin{align*}
\swt (w)&\le -1+\sum_{1\le 2j+1\le n-4 }\frac 1{\lambda^{2j+1} }+ \frac 1{\lambda^n}
<-1+\sum_{j\ge 0 }\frac 1{\lambda^{2j+1} }\\
&=-1+\frac {1/\lambda}{1-1/\lambda^2}=-1+\frac {\lambda}{\lambda^2-1}=-1+\frac {\lambda}{\lambda}=0.
\qedhere
\end{align*}
\end{proof}

\subsection{Almost self-similarity and related just infinite self-similar Lie algebra}

L.~Bartholdi introduced  the notion of self-similarity for Lie algebras~\cite{Bartholdi10}. 
A Lie algebra $L$ is called {\it self-similar} if it affords a homomorphism into the semidirect product 
$$\psi: L \to (R\otimes L ) \leftthreetimes \Der R ,$$
where $R$ is a commutative algebra and $\Der R$ its Lie algebra of derivations naturally acting on $R$. 
The notion of self-similarity for Lie algebras and a related notion of virtual endomorphism was further studied in~\cite{F-K-S}.

Now we discuss properties of the Lie algebra related to the Fibonacci Lie algebra $\tilde {\mathcal L}:=\Lie(\dd_1,v_2)\cong {\mathcal L}/A$, see Claim (2) of Theorem~\ref{T_L_decomp}.
\begin{lemma} \label{Lself}
$\tilde {\mathcal L}$ is self-similar.
\end{lemma}
\begin{proof}
Set $X: = K[t]/ (t^2)$, respective derivative denote as $\partial$. 
We consider $\tilde {\mathcal L}=\Lie(\dd_1,v_2)$ as a subalgebra of $\Der R$ for the latter we have the shift endomorphism $\tau$.
Define a homomorphism
\begin{align*}
&\psi : \tilde {\mathcal L}   \to (X \otimes \tau(\tilde {\mathcal L})) \leftthreetimes \Der X\qquad \text{by setting}\\
&\psi(\partial_1) :=\partial,\\
&\psi(v_2)=\psi(\partial_2+t_1v_3) :=1\otimes \partial_2 +t\otimes v_3=1\otimes \tau(\partial_1) +t\otimes \tau(v_2).%\qedhere
\end{align*}
It remains to identify $\tau(\tilde {\mathcal L})$ with $\tilde {\mathcal L}$.
\end{proof}
This Lie algebra  $\tilde {\mathcal L}$ was constructed as a self-similar one by L.~Bartholdi~\cite{Bartholdi10}.
It also appeared in~\cite{PeSh13fib} with another interesting property.
\begin{theorem}[{\cite[Theorem 7.8]{PeSh13fib}}] \label{Tjust}
The Lie algebra $\tilde {\mathcal L}= \Lie(\dd_1,v_2)$ is {\it just infinite dimensional}, i.e. its nontrivial ideals are of finite codimension.
\end{theorem}

\begin{lemma}
We have an "almost self-similarity" structure  on $\mathcal L$:
$$ \mathcal L\oplus \langle t_0v_2, t_0v_3\rangle_K
=\langle v_1\rangle _K \rightthreetimes ( K[t_0]/(t_0^2)) \otimes L_{(2)},  $$
where  $L_{(2)}=\Lie(v_2,v_3)=\tau(\mathcal L)\cong\mathcal L$, and the action is $v_1(t_0)=0$,
$v_1(v_2)=v_3$, $v_1(v_3)=t_0v_4=t_0[v_2,v_3]$.
\end{lemma}
\begin{proof}
Follows from decompositions of Lemma~\ref{LWrecc} and the proof of Theorem~\ref{T_L_decomp}.
\end{proof}

\subsection{Generating functions}
Using gradings of Lemma~\ref{Lhomogeneous} and Corollary~\ref{Chomogeneous},
we can define the {\em Hilbert series} in one and two variables for all our algebras
\begin{align*}
\mathcal H(\mathcal L,x,y) &:=\sum_{a+b\ge 1} \dim \mathcal L_{a,b} x^a y^b;\\
\mathcal H(\mathcal L,t) &:=\sum_{n\ge 1} \dim \mathcal L_{n} t^n=\mathcal H(\mathcal L,t,t).
\end{align*}
Similarly, we define the Hilbert series for multihomogeneous sets of monomials.
Sometimes, we shall omit variables and simply write $\mathcal H(\mathcal L)$.
\begin{theorem}\label{TreccLLL}
Let $p=2$ and $\mathcal L=\Lie(v_1,v_2)$.
Then the Hilbert series satisfy functional equations
\begin{align}
\mathcal H(W_{\le 2},x,y)&=x+y;\nonumber\\
\mathcal H(W_{\le n+1},x,y)&=\mathcal H(W_{\le n},y,xy)(1+x/y)-x^2,\qquad n\ge 2;\label{wn_recursion}\\
\mathcal H(\mathcal L,x,y)&=\mathcal H(\mathcal L,y,xy)(1+x/y)-x^2.\nonumber
\end{align}
\end{theorem}
\begin{proof}
Remark that $\Gr(t_0v_3)=\Gr(v_1^2)=2\Gr(v_1)=(2,0)$, so $\Gr(t_0)=(2,0)-\Gr(v_3)=(1,-1)$.
So, $\Gr(t_0v_2)=(1,-1)+(0,1)=(1,0)=\Gr (v_1)$.
Then $\mathcal H(t_0)=x/y$, $\mathcal H(t_0v_3)=x^2$, and $\mathcal H (t_0v_2)=\mathcal H(v_1)=x$.

Observe that $\tau$ yields the substitution $x:=y$ and $y:=xy$ into the right hand side.
Now, the result follows from the recursive relations of Lemma~\ref{LWrecc}.
\end{proof}

\begin{corollary}\label{CHrecursive}
By construction, $\mathcal H(W_{\le n},x,y)$ is the initial segment of the series $\mathcal H(\mathcal L,x,y)$ corresponding to the standard monomials of length at most $n$.
Starting with $\mathcal H(W_{\le 2},x,y)$ and iterating~\eqref{wn_recursion} we easily compute initial coefficients of $\mathcal H(\mathcal L,x,y)$.
The results of computations are shown on Fig.~\ref{Fig1}, Fig.~\ref{Fig2}, and Fig.~\ref{Fig3}.  
\end{corollary}

\begin{conjecture}\label{ConWN}
Monomials $W_N$ belong to two families of parallel lines, where $N\ge 1$ is fixed (see Fig.~\ref{Fig2}).
\end{conjecture}

%*********************************************************************************************
\section{Properties of  related algebras}
\subsection{Properties of associative hulls}
\begin{theorem} [{\cite{PeSh09,PeSh13fib}}] \label{TpropA}
Let $\ch K=p\in \{2,3\}$,  and $R:=K[t_i| i\ge 0]/(t_i^p| i\ge 0)$ the truncated polynomial ring.
Consider the Lie algebra $\mathcal L:=\Lie(v_1,v_2)$,
the restricted Lie algebra $\mathbf L:=\Lie_p(v_1,v_2)$, and
the associative hull $\mathbf A:=\Alg(v_1,v_2)\subset \End R$, all generated by $v_1,v_2$.
Set $\lambda:=(1+\sqrt 5)/2$. Then
\begin{enumerate}
 \item
  $\GKdim \mathbf L=\LGKdim \mathbf L= \GKdim \mathcal L=\LGKdim \mathcal L=\log_\lambda p$.
 \item
  $\GKdim \mathbf A=\LGKdim \mathbf A= 2 \log_\lambda p$.
 \item
  $\mathbf L$ has a nil-$p$-mapping.
 \item We put homogeneous components of the algebras on plane using the multidegree coordinates.
  Then the respective points of the lattice $\mathbb Z^2$ for $\mathcal L$, $\mathbf L$ and $\mathbf A $ lie in strips.
  The border lines for these stirps are determined by points of fixed superweight values.
 \item
  The points for PBW-monomials of the restricted enveloping algebra $u(\mathbf L)$ are bounded by a parabola-like curve,
  which in the weight coordinates $(\xi,\eta)$ is written as~\cite[Theorem~5.1]{PeSh09}:
\begin{equation*}%%\label{paraboloid}
|\eta|< C \xi^{\theta},\qquad \theta:=\frac{\lambda}{\lambda+1}=\log_{1+\sqrt 5}2\approx 0,5902,\quad C>0.
\end{equation*}
  \item Using geometric observations similar to that of Theorem~\ref{Tsum2},
  $\mathbf L$, $\mathbf A$, and the augmentation ideal of the restricted enveloping algebra $\mathbf u:=u_0(\mathbf L)$
  are direct sums of two locally nilpotent subalgebras
  $$\mathbf L=\mathbf L_+\oplus \mathbf L_-,\quad \mathbf A=\mathbf A_+\oplus \mathbf A_-,\quad \mathbf u=\mathbf u_+\oplus \mathbf u_-.$$
\end{enumerate}
\end{theorem}
\begin{conjecture}
Remark that $\theta$ for the parabola-like curve above is the same as $\theta$ for the intermediate growth of the universal enveloping algebra in  Corollary~\ref{Cinter} and~\eqref{boundU}.
It seems that such an equality should be valid for many other examples of fractal restricted Lie algebras and Lie superalgebras mentioned below.
\end{conjecture}

The lower central series, the derived series
for $\mathcal L$, $\mathbf L$, and $\tilde{\mathcal L}$ in case $p=2$ are described in~\cite{PeSh13fib}.
\medskip

In case of $p=0$, a little is known about $L=\Lie(v_1,v_2)\subset \Der R$.
The growth of $L$ is definitely greater than the polynomial growth (unpublished),
also $L$ is $\mathbb Z^2$-graded by the multidegree in the generators (Corollary~\ref{Cgrading}).
\begin{conjecture}
Let $\ch K=0$ and $R=K[t_i| i\ge 0 ]$. Determine the growth of the Lie algebra $L=\Lie(v_1,v_2)\subset \Der R$.
\end{conjecture}

\subsection{Poisson algebras}
A $K$-vector space $A$ is a {\it Poisson algebra} if it has two $K$-bilinear operations:
\begin{enumerate}
\item
$A$ is an associative commutative algebra with unit whose multiplication is denoted by $a\cdot b$, $a, b\in A$.
\item
$A$
is a Lie algebra whose product is denoted by the {\it Poisson bracket} $\{a, b\}$, where $a, b\in A$.
\item these two operations are related by the {\it Leibnitz rule}:
\begin{equation*}
\{a\cdot b, c\}=a\cdot\{b, c\}+\{a, c\}\cdot b,\qquad  a, b, c \in A.
\end{equation*}
\end{enumerate}
Let $L$ be a Lie algebra, $\{U_n| n\ge 0\}$ the natural filtration (i.e. by degree in $L$) of its universal enveloping algebra $U(L)$.
Consider the respective graded space supplied with the induced product $S(L):=\gr U(L)=\mathop{\oplus}\limits_{n=0}^\infty U_{n}/U_{n+1}$ (see~\cite{Dixmier}).
Define also a Poisson bracket by setting $\{v,w\}:=[v,w]$, $v,w\in L$,
and extending to the whole of $S(L)$ by linearity and the Leibnitz rule.
Then $S(L)$ ia a polynomial algebra which is turned into a Poisson algebra, called the {\it symmetric algebra} of $L$.
Let $L(X)$ be the free Lie algebra generated by a set $X$, then $S(L(X))$ is the free Poisson algebra~\cite{Shestakov93}.

More generally, let $A$ be an associative (super)algebra with unity generated by a set $X$.
Denote by $A^{(n,X)}$ the linear space of all monomials in $X$ of length at most $n$, where $n\ge 0$, and $A^{(0,X)}=\langle 1\rangle _K$.
Then $\{A^{(n,X)}\mid n\ge 0\}$ is a filtration and the associated graded algebra $\gr A$ is (super)commutative.
As above, $\gr A$ has a structure of a Poisson (super)algebra as well.

So in~\cite{PeSh18FracPJ}, a fractal Lie superalgebra $L$ was introduced and studied along with its associative hull $A$.
The filtration of $A$ determined by degree in the generating set yields a Poisson superalgebra $\gr A$.
This Poisson superalgebra and related Jordan superalgebras were studied in details in~\cite{PeSh18FracPJ}.

\subsection{Restricted Poisson algebras}
Now consider that $\ch K=p>0$. In case $p>2$  {\it restricted Poisson algebras} were defined in~\cite{BezKal07} and further studied in~\cite{BaoYeZhang17}.  
Roughly speaking, one adds a formal $p$th power under some naturally arising axioms. 
The notion of the restricted Poisson algebra can also be extended to the case $p=2$,
such algebras were defined and studied in~\cite{PeSh18FracPJ}.
In notations above, the associated graded space arising from the filtration by degree in the generating set of an associative algebra
has a structure of a restricted Poisson algebra. 

\begin{Remark}
The correct analogue of the Grigorchuk group~\cite{Grigorchuk80}
and Gupta-Sidki group~\cite{GuptaSidki83} in the world of associative algebras is not known jet.
We hope that the notion of the restricted Poisson algebra can help to resolve the problem
of nillity of the associative hull of the Fibonacci (restricted) Lie algebra as follows.
\end{Remark}

\begin{conjecture}
Let $p=2$ and $\mathbf A=\Alg(v_1,v_2)$,
the associative hull of the Fibonacci Lie algebra $\mathcal L=\Lie(v_1,v_2)\subset \Der R$. Then
\begin{enumerate}
\item
The restricted Poisson algebra $\gr \mathbf A$  arising from the filtration of $\mathbf A$ be degree in the generators is not nil.
The validity of this claim would immediately imply the following one.
\item
$\mathbf A$ is not nil.
\end{enumerate}
\end{conjecture}

\subsection{Similar constructions of restricted Lie algebras and Lie superalgebras}
In case of an arbitrary prime characteristic, Shestakov and Zelmanov
constructed an example of a 2-generated restricted Lie algebra with a nil $p$-mapping~\cite{ShZe08}.
An example of a $p$-generated  nil restricted Lie algebra $L$, characteristic $p$ being arbitrary, was studied in~\cite{PeShZe10}.
These infinite dimensional restricted Lie algebras and their restricted enveloping algebras as well, have
different decompositions into direct sums of two locally nilpotent subalgebras~\cite{PeShZe10} (similar to Theorem~\ref{TpropA}).
Unlike the Fibonacci restricted Lie algebra, in the latter
examples, elements of Lie algebras are  linear combinations of monomials;
to work with such linear combinations is sometimes an essential technical difficulty, see e.g.~\cite{ShZe08,PeShZe10,Pe20flies}.

A systematic approach to construct (restricted) Lie (super)algebras
having {\it good monomial bases} was developed due to the second Lie superalgebra introduced in~\cite{Pe16}.
Lie superalgebras over fields of arbitrary characteristic constructed in~\cite{Pe16,PeOtto}
can be viewed as natural analogues of the Grigorchuk and Gupta-Sidki groups in the class of Lie superalgebras.
This idea was extended to restricted Lie algebras, and now we have examples of nil restricted Lie algebras for any prime
with clear monomial bases and slow polynomial growth, including a quasi-linear growth~\cite{Pe17,Pe20clover}.
Also, the author constructed examples of nil restricted Lie algebras of intermediate oscillating growth in~\cite{Pe20flies}.
Next, along with a specially constructed just infinite fractal Lie superalgebra over an arbitrary field,
fractal  Poisson and Jordan superalgebras are constructed, supplying
analogues of the Grigorchuk and Gupta-Sidki groups in respective classes of algebras~\cite{PeSh18FracPJ}.

%***********************************************************************************
%\newpage
\section{Homology and Euler characteristic for Lie algebras of subexponential growth} % and Fibonacci Lie algebra}
\label{SHsubexp}

In this section we establish general results on homology and respective Euler characteristic of
infinite dimensional Lie algebras of subexponential growth (Theorem~\ref{TE} and Theorem~\ref{THsubexp}, the latter we call the infiniteness statement on homology) 
and consider applications of the results above to the Fibonacci Lie algebra (Theorem~\ref{Thomology} and Theorem~\ref{Tparaboloid}).
On homology of infinite dimensional Lie algebras see. e.g.~\cite{Hanlon96,KangKim99,Mill22}. 
\subsection{Homology of Lie algebras}
Let $L$ be a Lie algebra over a field $K$.
% and  $K$ denotes the trivial one-dimensional module.
Consider $K$-spaces of $n$-dimensional {\it chains} formed by the exterior powers $C_n(L):=\Lambda^n(L)$, $n\ge 0$,
and the {\it Chevalley-Eilenberg complex}
\begin{equation*}
\cdots \rightarrow\Lambda^{n+1}(L) \overset{d_{n+1}}{\longrightarrow}   \Lambda^n(L) \overset{d_n }{\longrightarrow}
\Lambda^{n-1}(L){\rightarrow}\ \cdots\ 
\overset{d_{2}}{\longrightarrow} \Lambda^{1}(L) \overset{d_1}{\longrightarrow} \Lambda^{0}(L)\overset{d_0}{\longrightarrow} 0,
\end{equation*}
where the {\it border differentials} $d_n: \Lambda^n(L)\rightarrow \Lambda^{n-1}(L)$ are defined by
\begin{equation*}
d_n(x_1\wedge \cdots \wedge x_n)=\sum_{s<t}(-1)^{s+t+1}[x_s,x_t ]
\wedge  x_1\wedge \cdots \widehat{x}_s \cdots \widehat {x}_t \cdots \wedge x_n,
\qquad x_i\in L,\quad n\ge 2,
\end{equation*}
and $d_1=d_0=0$.
It is well-known that the differentials are well-defined and $d_n\circ d_{n+1}=0$, hence $\Imm d_{n+1}\subset \Ker d_n $.
One defines the {\it homology groups with trivial coefficients} as $H_n(L):=\Ker d_n/\Imm d_{n+1}$ for $n\ge 0$.
One verifies that, $H_0(L)=K$, $H_1(L)=\Ker d_1/\Imm d_2\cong L/[L,L]$.

\subsection{Euler characteristic}
Assume that $L$ is a $\NO^k{\setminus} \{(0,\ldots,0)\}$-graded Lie algebra, where $k\ge 1$,  % namely, % we have
$$
L=\mathop{\oplus}\limits_{n_1,\ldots, n_k\ge 0} L_{n_1\ldots n_k}
=\mathop{\oplus}\limits_{n=1}^\infty \bigg( \mathop{\oplus}\limits_{n_1+\cdots+n_k= n} L_{n_1\ldots n_k}   \bigg).
$$
We assume that all components of the grading are finite dimensional and define the generating functions in $k$ variables and one variable
\begin{align*}
&\mathcal H(L,t_1,\ldots,t_k)  :=\sum_{n_1,\ldots, n_k\ge 0} \dim L_{n_1\ldots n_k} t_1^{n_1}\cdots t_k^{n_k};\\
&\mathcal H(L,t)  :=\sum_{n=1}^\infty \bigg(\sum_{n_1+\cdots+n_k=n} \dim L_{n_1\ldots n_k}\bigg ) t^{n}=\mathcal H(L,t,\ldots,t).
\end{align*}
In both cases we sometimes omit the variables. Observe that the chains and  homology groups naturally inherit the gradings and we can define respective generating functions. 

\begin{Remark}
Using approach of~\cite{Pe02},
the arguments of this section can be further extended to the case of $G$-graded Lie superalgebras, where $G$ is a semigroup under some assumptions. 
\end{Remark}

Identities~\eqref{euler}, \eqref{denominator} below are called in~\cite[(1.5), (1.8)]{KangKim99} as the {\it denominator identities}.
This Euler characteristic was considered in~\cite[Th. 4.1]{Hanlon96},\cite{KangKim99}.
\begin{lemma} 
Let $L$ be an $\NO^k{\setminus} \{(0,\ldots,0)\}$-graded Lie algebra.
In terms of the generating functions in $k$ variables (or one variable), the latter are omitted for brevity,  we have equality, both sides we call the {\bf Euler characteristic} of $L$:
\begin{equation}\label{euler}
{\mathbf E}(L):= \sum_{n=0}^\infty (-1)^n \mathcal H(\Lambda^n(L))=\sum_{n=0}^\infty (-1)^n \mathcal H(H_n(L)).
\end{equation}
\end{lemma}
\begin{proof}
Using equalities $\Imm d_{n}\cong \Lambda^n(L)/\Ker d_n$, $n\ge 0$, we consider respective generating functions
\begin{align*}
1=\mathcal H(\Lambda^0(L))&=\mathcal H(\Ker d_0),\\
\mathcal H(\Lambda^1(L))&=\mathcal H(\Ker d_1)+ \mathcal H(\Imm d_{1}),\\
&\cdots\\
\mathcal H(\Lambda^n(L))&=\mathcal H(\Ker d_n)+ \mathcal H(\Imm d_{n}),\\ 
&\cdots
%\mathcal H(\Lambda^{3}(L))&= \mathcal H(\Ker d_{3})+ \mathcal H(\Imm d_{3}), \\
%&\cdots
\end{align*}
Since $H_n(L)=\Ker d_n/\Imm d_{n+1}$, alternating summation of equations above yields equality~\eqref{euler}.
Remark that both sides in~\eqref{euler} converge in topology of formal power series because both $\mathcal H(\Lambda^n(L))$ and $\mathcal H(H_n(L))$
consist of terms of total degree at least $n$.
\end{proof}

\begin{lemma} \label{Leuler2}
In notations above, the left hand side of~\eqref{euler} (i.e. the Euler characteristic)  is equal to
\begin{equation}\label{denominator}
\sum_{n=0}^\infty (-1)^n \mathcal H(\Lambda^n(L),t_1,\ldots,t_k)
=\prod_{n_1,\ldots,n_k\ge 0} (1-t_1^{n_1}\cdots t_k^{n_k})^{\dim L_{n_1\ldots n_k}}.
\end{equation}
\end{lemma}
\begin{proof}
Let us prove the statement in case of $\N$-grading $L=\mathop{\oplus}\limits_{n=1}^\infty L_n$  and series in one variable, namely that
\begin{equation*}
\sum_{n=0}^\infty (-1)^n \mathcal H(\Lambda^n(L),t)=\prod_{n=1}^\infty (1-t^n)^{\dim L_{n}}.
\end{equation*} 
Let $\{w_j|j\in\N\}$ be a homogeneous basis of $L=\mathop{\oplus}\limits_{n\ge 1}L_n$.
Denote by $\{m_j| j\in\N\}$ indices of the respective components, i.e. $w_j\in L_{m_j}$, $j\ge 1$.
Then
\begin{equation}\label{three}
\sum_{n=0}^\infty (-1)^n \mathcal H(\Lambda^n(L),t)=\prod_{j= 1}^\infty (1-t^{m_j})=\prod_{n= 1}^\infty (1-t^{n})^{\dim L_n}.
\end{equation}
Indeed, let us prove the first equality~\eqref{three}.
Consider a product $(-t^{m_{i_1}})(-t^{m_{i_2}})\cdots (-t^{m_{i_s}}) $, where $i_1<\cdots< i_s$,
arising from the expansion of the middle part~\eqref{three}.
Then this term bijectively corresponds to the basis element
$w_{i_1}\wedge w_{i_2}\wedge\cdots \wedge  w_{i_s}\in  \Lambda^s(L)$, 
the latter yielding the term $(-1)^s t^{m_{i_1}+\cdots+ m_{i_s}}$ of $(-1)^s\mathcal H(\Lambda^s(L),t)$  in the left hand side of~\eqref{three}. 
The second equality~\eqref{three} is obtained by collecting $w_j$ with the same degree $n=m_j$.
The case of $k$ variables is similar.
\end{proof}

\subsection{Properties of Euler characteristic}
We introduce an $m$-{\it dilatation} of a formal power series $f(t_1,\ldots,t_k)$: 
\begin{equation*}
f^{[m]}(t_1,\ldots,t_k):=f(t_1^m,\ldots, t_k^m), \qquad m\in \N.
\end{equation*}

\begin{lemma}\label{Leuler3}
In notations above, the Euler characteristic equals
\begin{align*}
{\mathbf E}(L)&=\exp\bigg(-\sum_{m=1}^\infty \frac 1m \mathcal H^{[m]}(L)\bigg).
\end{align*}
\end{lemma}
\begin{proof}
Consider the case of one variable. By Lemma~\eqref{Leuler2}
\begin{align*}
{\mathbf E}(L,t)&= \prod_{n=1}^\infty (1-t^n)^{\dim L_{n}}=\exp\bigg(\sum_{n=1}^\infty \dim L_n\ln (1-t^n)\bigg)\\
&=\exp\bigg(- \sum_{n=1}^\infty \dim L_n \sum_{m=1}^\infty \frac{t^{mn}}m\bigg)
=\exp\bigg(- \sum_{m=1}^\infty \frac 1m \sum_{n=1}^\infty \dim L_n t^{mn}\bigg)\\
&=\exp\bigg(-\sum_{m=1}^\infty \frac 1m \mathcal H(L,t^m)\bigg)
=\exp\bigg(-\sum_{m=1}^\infty \frac 1m \mathcal H^{[m]}(L,t)\bigg).
\qedhere
\end{align*}
\end{proof}

Next, consider the operator ${\mathcal E}$ on power series with natural coefficients
%$\phi(t)=\sum_{n=1}^\infty b_n t^n$ and $\psi(t)=\sum_{n,m\ge 0}^\infty b_{n,m} x^ny^m$ as
(see~\cite{Pe99int,Pe00,Pe00Springer,Pe02}): 
\begin{equation}\label{seriesUL}
\begin{split}
{\mathcal E}\,:\ &\phi(t)=\sum_{n=1}^\infty b_n t^n\quad  \longmapsto\quad
   {\mathcal E}(\phi(t)): =\prod_{n=1}^\infty \frac 1{(1-t^n)^{b_n}},%%%=\sum_{n=0}^\infty a_n t^n,
   \\
{\mathcal E}\,:\ &\phi(t_1,\ldots,t_k)=\!\!\sum_{n_1,\ldots,n_k\ge 0}\!\! b_{n_1\ldots n_k} t_1^{n_1}\cdots t_k^{n_k}\quad \longmapsto\\
   &\qquad\qquad {\mathcal E}(\phi(t_1,\ldots,t_k)): =\prod_{n_1,\ldots,n_k\ge 0} \frac 1{(1-t_1^{n_1}\cdots t_k^{n_k})^{b_{n_1\ldots n_k}}}.
   %%%=\sum_{n_1,\ldots,n_k\ge 0} a_{n_1\ldots n_k} t_1^{n_1}\cdots t_k^{n_k}.
\end{split}
\end{equation}
\begin{lemma} [\cite{Pe99int,Pe02,Ufn}]\label{Leuler4}
Let $L$ be an $\NO^k{\setminus} \{(0,\ldots,0)\}$-graded Lie algebra  and  $U(L)$ its universal enveloping algebra with the induced grading.
Then the respective generating functions satisfy 
%%$ {\mathcal H}(L,t)=\sum_{n=1}^\infty b_n t^n$, and
%% ${\mathcal H}_X(U(L),t)=\sum_{n=0}^\infty a_n t^n$. In case of $\NO^2$-grading assumptions are similar. Then
$${\mathcal H}(U(L))={\mathcal E}({\mathcal H}(L)).$$
\end{lemma}
\begin{lemma}[\cite{Pe99int,Pe00,Pe00Springer,Pe02}]\label{Leuler5}
In notations above, the operator $\mathcal E$ can be computed as
\begin{align*}
{\mathcal E}(\phi(t_1,\ldots,t_k))=\exp\bigg(\sum_{m=1}^\infty \frac 1m \mathcal \phi^{[m]}(t_1,\ldots,t_k)\bigg).
\end{align*}
\end{lemma}
\begin{proof} The computations are similar to that of Lemma~\ref{Leuler3} and using~\eqref{seriesUL}. 
\end{proof}

Now we get the following result, that we did not find in the literature.
\begin{theorem}\label{TE}
Let $L$ be an $\mathbb N_0^k{\setminus}\{(0,\ldots,0)\}$-graded Lie algebra. Then the Euler characteristic equals
\begin{equation}\label{formulaE}
{\mathbf E}(L,t_1,\ldots,t_k)=\frac 1{\mathcal E(\mathcal H(L,t_1,\ldots,t_k))}=\frac 1{\mathcal H(U(L),t_1,\ldots,t_k)}.
\end{equation}
\end{theorem}
\begin{proof}
Follows from Lemmas~\ref{Leuler3}, \ref{Leuler4}, and \ref{Leuler5}.
\end{proof}

We shall use the following particular case.
\begin{lemma}\label{Lsubsexp}
Let $L=\mathop{\oplus}\limits_{n=1}^N L_n$ be a finite dimensional $\N$-graded Lie algebra with the generating function $\mathcal H(L,t)=\sum\limits_{n=1}^N b_n t^n$, where $b_n=\dim L_n$.
Then its Euler characteristic is a polynomial of the form
$$ \mathbf E(L,t)=\prod_{n=1}^N (1-t^n)^{b_n}=(1-t)^{\dim L} q(t),$$ 
where $q(t)\in \Z[t]$,  $q(1)=\prod_{n=2}^N n^{b_n}\ne 0$. 
\end{lemma}
\begin{proof}
Follows from Lemma~\ref{Leuler2} and formula \eqref{seriesUL}.
\end{proof}

\subsection{Euler characteristic and homology of Lie algebras of subexponential growth}
A sequence $\{a_n\in\Z\mid n\ge 0 \}$ is said of {\it subexponential growth} iff $\limsup\limits_{n\to\infty}\sqrt[n]{|a_n|}=1$, 
the equivalent condition is  that the generating function $f(z):=\sum_{n=0}^\infty a_n z^n$, $z\in \C$, has the radius of convergence 1.

Now we generalize Lemma~\ref{Lsubsexp} to an infinite dimensional case.
The coefficients of a Euler series can be negative (see e.g. Fig.~\ref{Fig4}), nevertheless we have the following fact.
\begin{lemma}\label{LE}
Let $L$ be a finitely generated $\mathbb N_0^k{\setminus}\{(0,\ldots,0)\}$-graded Lie algebra of subexponential growth. Then
\begin{enumerate}
\item $1\ge \mathbf E(L,t_1,\ldots, t_k)>0$ for all\quad $0\le t_i<1$,\ $1\le i\le k$.
\item $\mathbf E(L,t_1,\ldots, t_k)$ is decreasing with respect to all variables $t_1,\ldots,t_k\in [0,1)$.
\item $\lim\limits_{t\to 1-0} \mathbf E(L,t)=0 $.
\item $\mathbf E(L,z_1,\ldots, z_k)$ converges absolutely at any point $(z_1,\ldots, z_k)$ where $z_i\in \C$, $|z_i|<1$, $1\le i\le k$.
\item $\mathbf E(L,z)$ converges absolutely at least in the open unit disc.
\end{enumerate}
\end{lemma}
\begin{proof} 1) Let $t:=\max\{t_i\mid 1\le i\le k\}$. Then $0\le t<1$.
The subexponential growth of $L$ implies a subexponential growth of its universal enveloping algebra $U(L)$,
which is equivalent to the fact that the radius of convergence of the series $\mathcal H(U(L),t)$ is equal to 1~\cite{Smi}. 
Since $\mathcal H(U(L),t_1,\ldots,t_k)$ has nonnegative coefficients, using Theorem~\ref{TE} we get
\begin{align*}
1= \mathcal H (U(L),0,\ldots, 0) &\le \mathcal H (U(L),t_1,\ldots, t_k)  \le \mathcal H (U(L),t,\ldots, t)=\mathcal H (U(L),t)<\infty;\\
1 \ge \mathbf E(L,t_1,\ldots,t_k) &=\frac {1}{\mathcal H (U(L),t_1,\ldots, t_k)}>0;\\
\lim\limits_{t\to 1-0} \mathbf E(L,t) &= \lim\limits_{t\to 1-0}  \frac 1{\mathcal H (U(L),t) }=\frac 1{\dim U(L)}=0;
\end{align*}
thus proving 1) and 3). Claim 2) follows from~\eqref{formulaE}.   

4) We compare coefficients at ${(n_1,\ldots,n_k)}$-components in~\eqref{euler}, where $n_1,\ldots,n_k\ge 0$, and get  inequalities
\begin{equation*}
|({\mathbf E}(L))_{n_1\ldots n_k}|\le \sum_{n=0}^\infty \dim \Lambda^n_{n_1\ldots n_k}(L) =\dim \Lambda_{n_1\ldots n_k}(L) \le \dim U_{n_1\ldots n_k}(L).
\end{equation*}
Consider $t:=\max\{|z_i|  \mid 1\le i\le k\}$, then $t<1$.
Then
\begin{equation*}
|\mathbf E(L,z_1,\ldots,z_k)|\le \mathcal H(U(L),|z_1|,\ldots,|z_k|)\le \mathcal H(U(L),t)<\infty.  
\end{equation*}
Thus we proved Claims 4,5).
\end{proof}

We establish the following infiniteness statement on homology.

\begin{theorem}\label{THsubexp}
Let $L$ be an infinite dimensional finitely generated $\mathbb N$-graded Lie algebra of subexponential growth. Then
\begin{enumerate}
\item the Euler characteristic $\mathbf E(L,z)$ satisfies estimates close to 1:
\begin{equation}
0<{\mathbf E}(L,t)\le \exp\bigg(\frac {-1/2}{1-t}\bigg),\qquad t\in (1-\epsilon,1). \label{up} 
\end{equation} 
\item the Euler characteristic $\mathbf E(L,z)$ has the radius of convergence one, where $z_0=1$ is a singularity;
\item the Euler characteristic $\mathbf E(L,z)$ is not a polynomial function;
\item
$\displaystyle \sum_{n=2}^\infty \dim H_n(L)=\infty$. %% \quad (infiniteness statement on homology groups).
\end{enumerate}
\end{theorem}
\begin{proof} By Claim 5) of Lemma~\ref{LE}, $f(z):=\mathbf E(L,z)$ converges at least in the disk $D:=\{z\in\C\mid |z|<1\}$.
Consider the border $\partial D:=\{z\in \C\mid |z|=1\}$.
Recall that a point $z_0\in \partial D$ is called {\it regular} provided that there exists an {\it analytic continuation} in a vicinity of $z_0$,
namely, there exists $\epsilon>0$ and analytic function $g(z)$ in $D_1:=\{z\in \C\mid |z-z_0|<\epsilon\}$ such that $f(z)\equiv g(z)|_{D\cap D_1}$. 
Otherwise $z_0\in \partial D$ is called a {\it singularity}.
Let us show that the point $z_0=1$ is a singularity.
By  way of contradiction, assume that there exists an analytic function $g(z)$  in $D_1:=\{z\in \C\mid |z-1|<\epsilon\}$, $0<\epsilon<1/2$, such that $f(z)\equiv g(z)|_{D\cap D_1}$. 
Let $z_0=1$ be a zero of $g(z)$ of order $m$, where $m\ge0$. Thus, we have $g(z)=(1-z)^m g_0(z)$, where $g_0(z)$ is analytic in $D_1$ and $g_0(1)=C\ne 0$.
Consider both functions on the real segment $t\in (1-\epsilon,1)$.
By Claim 1) of Lemma~\ref{LE}, $(1-t)^m g_0(t)=f(t)>0$ for $t\in (1-\epsilon,1)$, hence $C=g_0(1)=\lim\limits_{t\to 1-0} f(t)(1-t)^{-m}>0$.
We use Lemma~\ref{Leuler3} and that all $\N$-components of $L$ are at least one dimensional
\begin{multline*}
f(t)={\mathbf E}(L,t)=\exp\bigg(-\sum_{n=1}^\infty \frac 1n \mathcal H(L, t^n)\bigg)\le \exp\big(-\mathcal H(L, t)\big)\\
\le \exp\bigg(-\sum_{n=1}^\infty t^n\bigg)= \exp\bigg(-\frac t{1-t}\bigg)\le \exp\bigg(\frac {-1/2}{1-t}\bigg),\qquad t\in (1-\epsilon,1). %%\label{up} 
\end{multline*}
Claim 1 is proved.
Now we see that $f(z)$ tends to zero too fast while $t\to 1-0$. Formally, using estimate~\eqref{up} we have
\begin{align*}
0&<C=g_0(1)=\lim_{z\to 1}\frac {g(z)}{(1-z)^m} =\lim_{t\to 1-0}\frac {f(t)}{(1-t)^m} \\  
& \le \lim_{t\to 1-0} \frac 1{(1-t)^m}\exp\bigg(-\frac {1/2}{1-t}\bigg)= \Big|_{y:=1/(1-t)}
 =\lim_{y\to +\infty} \frac{y^m} {\exp(y/2)}=0,
\end{align*}
a contradiction. Hence, $z_0=1$ is a singularity on $\partial D$ and we conclude that $\mathbf E(L,z)$ has the radius of convergence one, Claim 2 is proved.
Since a polynomial function is an entire function, $\mathbf E(L,z)$ has infinitely many nonzero coefficients, thus proving Claim 3.

Let us prove Claim 4. We know that $\dim H_1(L)$ is equal to the number of generators, so $\dim H_1(L)<\infty$.
By way of contradiction, assume that $\sum_{n=2}^\infty \dim H_n(L)$ is a finite number. 
Looking at the right hand side of equation~\eqref{euler} we conclude that $\mathbf E(L,z)$ is a polynomial, a contradiction.
Theorem is proved.
\end{proof}

\begin{corollary}
Let $L$ be one of the simple infinite dimensional Cartan type Lie algebras $\mathbf W_n$, $\mathbf S_n$, $\mathbf H_n$, $\mathbf K_n$, where $n\ge 1$,
over a field of characteristic zero. Then
$$
\displaystyle \sum_{n=3}^\infty \dim H_n(L)=\infty. %% \quad (infiniteness statement on homology groups).
$$ 
\end{corollary}
\begin{proof}
We apply Theorem~\ref{THsubexp} and use that $\dim H_2(L)< \infty$, in particular these Lie algebras are finitely presented, see~\cite{LeiPol}. 
\end{proof}

\subsection{Euler characteristic and Homology groups of Fibonacci Lie algebra and their positions on plane}
Now we return to the Fibonacci Lie algebra $\mathcal L=\Lie(v_1,v_2)$, where $p=2$.
Clearly,  $H_1(\mathcal L)\cong \langle v_1,v_2\rangle_K$.
Now we present results on homology groups of $\mathcal L$ following from our results above.

\begin{theorem}\label{Thomology}
Let $p=2$, $\mathcal L=\Lie(v_1,v_2)$ the Fibonacci Lie algebra, and $W$ its basis~\eqref{standard}.
Let $u=w_1\wedge \cdots \wedge w_n\in\Lambda^n(\mathcal L)=C_n(\mathcal L)$, where $w_i\in W$ (or consider that $u\in H_n(\mathcal L)$), $n\ge 1$. Then
\begin{enumerate}
\item
We have a natural extension of the weight functions onto the spaces of chains and homology groups:
$$\Wt(u):=\Wt w_1+\cdots +\Wt w_n\in\R^2.$$
\item One has a natural extension of the multidegree
$$ \Gr(u):=\Gr w_1+\cdots+ \Gr w_n\in\Z^2.$$
\item We get $\Z^2$-gradings  of the chain spaces and homology groups, e.g.
$$H_n(\mathcal L)=\mathop{\oplus}\limits_{(a,b)\in \Z^2} H_{n,(a,b)} (\mathcal L),\qquad n\ge 1,$$
where $H_{n,(a,b)} (\mathcal L)$ is spanned by elements of  $H_n(\mathcal L)$ of the multidegree $(a,b)\in\Z^2$.
\item Assume that $\{w_1,\ldots, w_n\}\subset W_{\le m}$ and at least one $w_i$ belongs to $W_m$, where $m\ge 1$. Then
$$ \lambda^{m-1} <\wt u\le n \lambda^m. $$
\item Consider the multidegree coordinates $\Gr(u)=(x,y)\in\Z^2$.
Then respective points for $H_n(\mathcal L)$ (or $C_n(\mathcal L)$) belong to the strip in terms of  the superweight or multidegree coordinates
\begin{align*}
-\lambda n <\eta&=\swt u< n;\\
-\lambda^3n +\lambda x& < y <\lambda x+\lambda^2n.
\end{align*}
\item Upper bounds on the growth  (defined via the $\N$-grading) of homology groups hold
$$\GKdim H_n(\mathcal L)\le n\log_\lambda 2.$$ 
\item We have infiniteness statement
$$
\sum_{n=2}^\infty \dim H_n(\mathcal L)=\infty.
$$
\end{enumerate}
\end{theorem}
\begin{proof}
We use that the weight functions are additive and differentials keep them and proceed as above
using estimates of Lemma~\ref{LwtW}, Corollary~\ref{CstWn}, and the relation between two systems of coordinates~\eqref{newcoord}.
The last two claims  follow from Theorem~\ref{TgrowthP} and Theorem~\ref{THsubexp}.
\end{proof}

\begin{theorem}\label{Tparaboloid}
  Let $\mathcal L$ be the Fibonacci Lie algebra, $\ch K=2$. 
  The points for nonzero monomials of the Eulerian characteristic $\mathbf E(\mathcal L,x,y)$
  and nonzero monomials of the series for the homology groups $\{H_n(\mathcal L)\mid  n\ge 1\}$
  are bounded by a parabola-like curve on plane, which in terms of the weight coordinates $(\xi,\eta)$ is as follows (shown on Fig.~\ref{Fig4}):
\begin{equation}
|\eta|< C \xi^{\theta},\qquad \theta:=\log_{1+\sqrt 5}2\approx 0,5902.
\end{equation}
\end{theorem}
\begin{proof}
Monomials for the exterior powers
$\Lambda^n(\mathcal L)$ are $w_{i_1}\wedge\cdots \wedge w_{i_n}$, where, $w_{i_1}<\cdots < w_{i_n}$, and
$w_{i_j}\in W$ are basis elements~\eqref{standard} of $\mathcal L$.
The basis of $\mathbf u(\mathbf L)$ is "a little bigger", because the ordered PBW-monomials are written
via a little bigger basis $\tilde W$~\eqref{basisLL} of the $\mathbf L$.
The result follows from Theorem~\ref{TpropA}, Claim 5.
\end{proof}

Using asymptotic of the universal enveloping algebra described by Theorem~\ref{TgrowthU} by applying results and methods of~\cite{Pe99int}, 
we expect to specify how the Euler characteristic tends to zero while approaching the singularity $z_0=1$, thus strengthening 
Claim 3) of Lemma~\ref{LE} and the upper bound~\eqref{up}.  
%The upper bound below directly follows from Theorem~\ref{TgrowthU} and~\cite{Pe99int}, additional computations are needed to establish the lower bound only.
This result might have applications to the homology.  
\begin{conjecture} 
Let $\mathcal L$ be the Fibonacci Lie algebra. There exist $C_1,C_2>0$ such that  the following asymptotic for the Euler characteristic holds:
\begin{equation*}
\exp\bigg( -\frac{ C_1}{(1-t)^{\log_\lambda 2}} \bigg)\le \mathbf E(\mathcal L,t)\le \exp\bigg( -\frac{ C_2}{(1-t)^{\log_\lambda 2}} \bigg),\qquad t\to 1-0.
\end{equation*} 
\end{conjecture}

Denote the multihomogeneous coordinates as $\Gr(w)=(\Gr_1(w),\Gr_2(w))$, where $w$ is a monomial of the Fibonacci Lie algebra $\mathcal L$.
\begin{lemma} The Euler characteristic of the Fibonacci Lie algebra $\mathcal L=\Lie(v_1,v_2)$ 
with respect to its $\N_0^2\setminus \{(0,0)\}$-grading by the multidegree in the generators  is equal to
$$ \mathbf E(\mathcal L,x,y)=\lim_{n\to \infty} \prod_{w\in W_{\le n}}(1-x^{\Gr_1(w)}y^{\Gr_2(w)}), $$
where the limit is in terms of the formal power series topology.
\end{lemma}
\begin{proof} 
Follows from ~\eqref{euler}, \eqref{denominator} and Lemma~\ref{CHrecursive}.
\end{proof}
Using the recursive construction of the basis $W$ given by Claim 1 of Corollary~\ref{CW}, we compute the initial part of the Euler characteristic depicted on plane on Fig.~\ref{Fig4}.
\begin{Remark}
The behaviour of the coefficients of the Euler characteristic $\mathbf E(\mathcal L,x,y)$  seen on Fig.~\ref{Fig4} is really chaotic. 
\end{Remark}

\section{Infinite presentation of $\mathcal L$}
\subsection{Finitely presented Lie algebras}
It is well known that finite dimensional simple Lie algebras
over an algebraically closed field of characteristic zero are finitely presented,
namely, they are determined by now classical Serre's relations.
More generally,  the Kac-Moody algebras are presented
by similar finitely many relations in terms of the Cartan matrix~\cite{Kac}.

I.~Stewart proved that the simple infinite dimensional Witt Lie algebra $W_1$ in case of characteristic zero is finitely presented~\cite{Stewart75}.
V.~Ufnarovskiy remarked that the positive part of $W_1$ is also finitely presented~\cite{Ufn80}.
More generally, Leites and Poletaeva showed that the simple infinite dimensional Cartan type Lie algebras $W_n,S_n,K_n,H_n$ in characteristic zero
are finitely presented~\cite{LeiPol}.
Presentations of Lie algebras and close objects are extensively studied in terms of the Gröbner-Shirshov bases,
in particular, this approach was applied to simple finite dimensional Lie algebras and Kac-Moody algebras,
see e.g.~\cite{BokChen07,BokChen14,BokChen18,BokKle96} and references therein.
Finitely presented Lie algebras and intermediate growth of their universal enveloping algebras were studied in~\cite{Kos15, Kos17}.
Futorny, Kochloukova and Sidki constructed in \cite[Thm. D]{F-K-S} examples
of finitely presented, infinite dimensional, metabelian, self-similar Lie algebras. 
The question when metabelian Lie algebras are finitely presented was solved by Bryant and Groves~\cite{B-G1,B-G2}.
%Finite presentation of Lie superalgebras is studied in~\cite{KochPe2}.

\subsection{Infinite presentation of the Fibonacci Lie algebra $\mathcal L$ and homology groups}
The fact that the Grigorchuk group is infinitely presented was proved by Grigorchuk~\cite{Grig99}.
The Gupta-Sidki group is also infinitely presented~\cite{Sidki87}.
\medskip

The following conjecture was initially suggested in~\cite{PeSh13fib}.
\begin{conjecture}
Let $p=2$. The Fibonacci restricted Lie algebra  $\mathbf L=\Lie_p(v_1,v_2)$ is infinitely presented.  
\end{conjecture}

Recently, using methods of homological algebra a similar result was established which first claim is stronger than the last claim of Theorem~\ref{THsubexp}.
\begin{theorem}[\cite{KochPe}]
Let $p=2$, consider the Fibonacci Lie algebra $\mathcal L=\Lie(v_1,v_2)$. Then
\begin{enumerate}
\item $\dim H_2(\mathcal L)=\infty$;
\item $\mathcal L$ is infinitely presented; 
\item the related just infinite self-similar Lie algebra $\tilde L=\mathcal L/A\cong \Lie(\dd_1,v_2)$ 
(see Theorem~\ref{T_L_decomp}, Lemma~\ref{Lself}, and Theorem~\ref{Tjust}) is also infinitely presented.
\end{enumerate}
\end{theorem}

Looking at Fig.~\ref{Fig4}, it seems that the nontrivial monomials of the Euler characteristic $\mathbf E(\mathcal L, x,y)$
essentially fill the paraboloid-shaped area specified in Theorem~\ref{Tparaboloid}.
On the other hand, points that depict any fixed homology group $H_n(\mathcal L)$, $n\ge 1$, belong only to the respective strip like the shaded one on Fig.~\ref{Fig4} 
(Claim 5 of Theorem~\ref{Thomology}). 
This motivates us to formulate the following conjecture. 

\begin{conjecture}
Let $p=2$, consider the Fibonacci Lie algebra $\mathcal L=\Lie(v_1,v_2)$. Then also
$$\dim H_n(\mathcal L)=\infty,\qquad n\ge 3.$$
\end{conjecture}

\subsection{Relations of $\mathcal L=\Lie(v_1,v_2)$ and $\mathbf L=\Lie_p(v_1,v_2)$}
Initially, a countable set of relations of the Grigorchuk group was found by Lysionok~\cite{Lys85},
it is obtained from a finite set by applying iterations of an embedding endomorphism.
Later, such presentations were called finite {\it endomorphic presentations} or $L$-{\it presentations} by Bartholdi~\cite{Bartholdi03}.
We conjecture that a similar situation holds for Fibonacci Lie algebras.
We suggest the following  list of relations.

\begin{lemma}\label{Lrelations}
Let $p=2$, $\mathcal L=\Lie(v_1,v_2)$ and
$\tau:\mathcal L\to \mathcal L$ the shift endomorphism.
Then $\mathcal L$ satisfies the following relations and all their shifts by $\tau^k$, $k\ge 1$:
\begin{enumerate}
\item  $[v_2,v_1^3]=0$;
\item $[v_2,v_1^2,v_2^2]=0$;
\item $[v_1,v_2^4]=0$, which is a $\tau$-image of the first relation;
\item the Fibonacci restricted Lie algebra $\mathbf L=\Lie_p(v_1,v_2)$ satisfies $v_1^4=0$.
\end{enumerate}
\end{lemma}
\begin{proof}
1) $[v_2,v_1^3]=[[v_2,v_1],v_1^2]=[v_3,t_0v_3]=0$.
 
2) $[v_2,v_1^2,v_2^2]=[v_2,t_0v_3,t_1v_4]=[t_0v_4,t_1v_4]=0$.

3) $\tau([v_2,v_1^3])=[v_3,v_2^3]=[[v_1,v_2],v_2^3]=[v_1,v_2^4]$.

4) $v_1^4=(t_0v_3)^2=0$. 
\end{proof}

\begin{conjecture}
Let $p=2$.
The Fibonacci Lie algebra $\mathcal L$ and the Fibonacci restricted Lie algebra $\mathbf L$ 
are defined by relations given in Lemma~\ref{Lrelations}.
\end{conjecture}

Confirming this conjecture we have the following observation.
\begin{lemma}
The relations of $\mathcal L$ of degree at most 7 follow from the relations
$$[v_2,v_1^3]=0, \quad [v_2,v_1^2,v_2^2]=0,\quad [v_1,v_2^4]=0.$$
\end{lemma}
\begin{proof} A direct check that we omit.
\end{proof}

\section*{Statements and Declarations}
\begin{itemize}
\item {\bf Funding:} partially supported by grant 
Funda\c{c}\~ao de Amparo \`a Pesquisa do Estado de S\~ao Paulo, FAPESP 2022/10889-4.  
\item
the author is grateful to the State University of Campinas
and prof. Dessislava Kochloukova 
for hospitality when this research was done.
%\item {\bf conflict of interests:} none
%\item {\bf Data Availability declaration:}\\ The GAP and METAPOST computation data are available on\\
%\url{https://osf.io/qb7wp/?view_only=3f0d191aa2e24cd0a4fabecf9dda3dfb}
\end{itemize}

%-------------------------------------------------------------------------------------------------------
\begin{figure}[h]
\centering
\caption{
Standard monomials $W$ of the Fibonacci Lie algebra $\mathcal L=\Lie(v_1,v_2)$, $p=2$. Point $(a,b)$ is green --- all $W_n\cap  \mathcal L_{(a,b)}$ are of form $[v_{n-2},W_{n-1}]$, 
blue --- all are of form $[v_{n-3},W_{n-1}]$, proportionally mixed color --- monomials of both types. The pivot elements are red.
The area of circles is proportional to the number of monomials, indicated also by numbers.
The red lines pass through the pivot elements.
Monomials of the abelian ideal $A$ are below axis $\xi$.}
\label{Fig1} 
\includegraphics [width=0.95\textwidth]{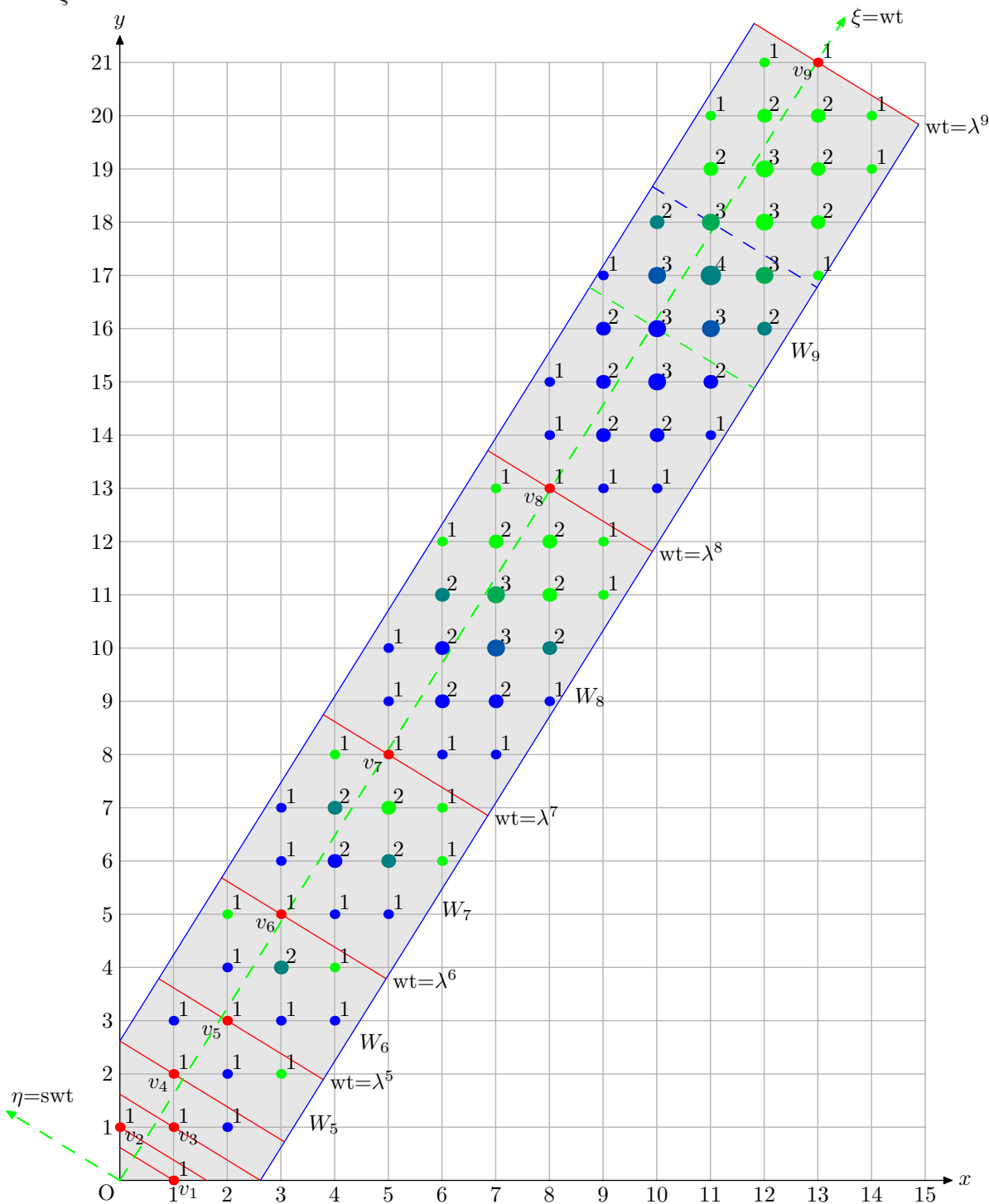}
\end{figure}
%----------------------------------------------------------------------------
\begin{figure}[h]
\centering
\caption{Standard monomials $W_N$, where $N=15$, are shown in respective rectangle
$\{(\xi,\eta) \mid \lambda^{N-1} < \xi \le\lambda^N, \, -\lambda <  \eta< 1\}$.
It seems that points belong to {\it two families of parallel lines} (Conjecture~\ref{ConWN})}
\label{Fig2} 
\includegraphics [width=0.95\textwidth]{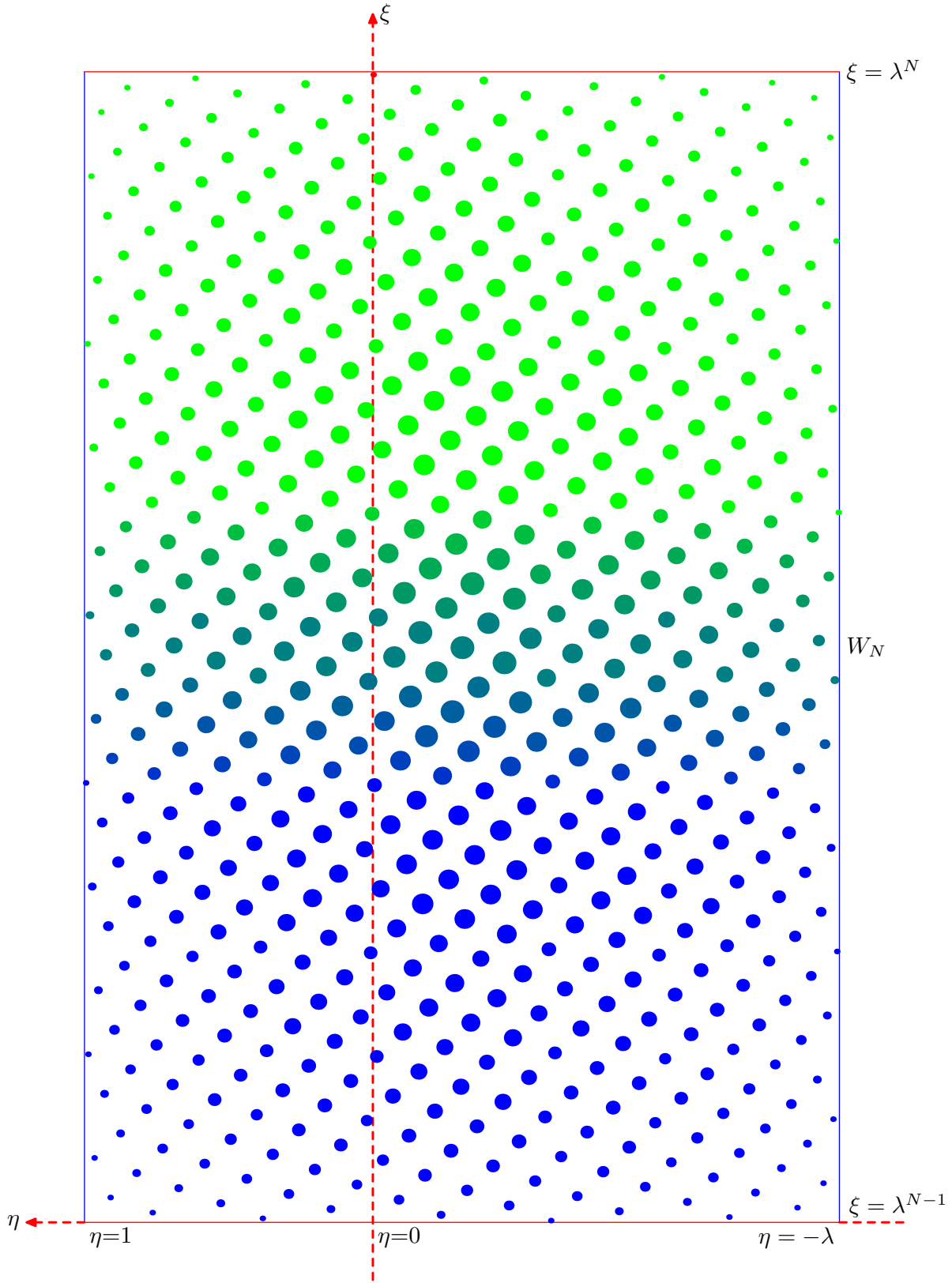}
\end{figure}
%----------------------------------------------------------------------------
\begin{figure}[h]
\centering
\caption{Standard monomials of $W_{\le 15}$ in their rectangles are shown together in one rectangle.
The pattern of two families of parallel lines observed for a fixed  $W_N$ (Fig.~\ref{Fig2}) disappears}
\label{Fig3} 
\includegraphics [width=0.95\textwidth]{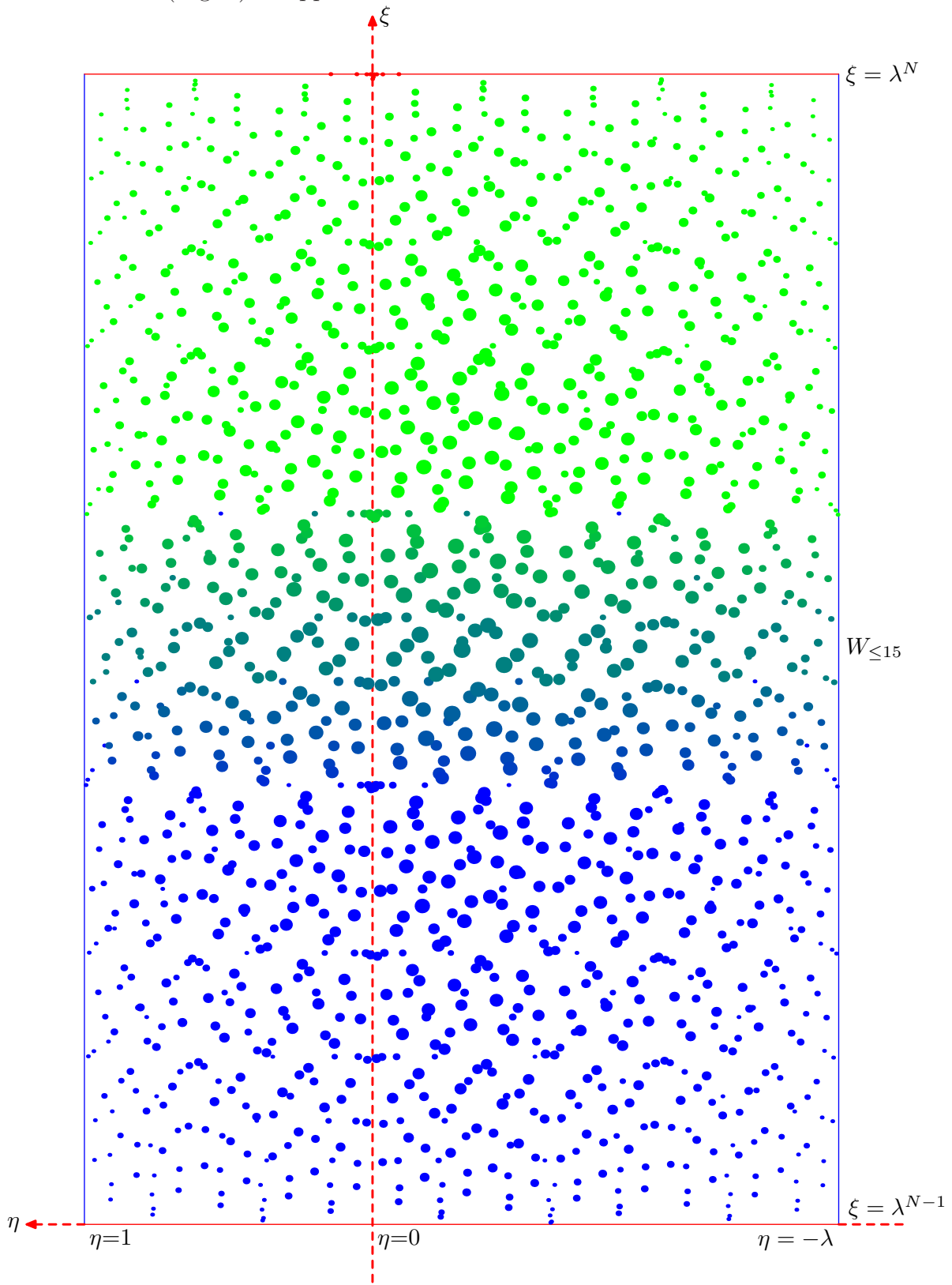}
\end{figure}
%----------------------------------------------------------------------------
\begin{figure}[h]
\centering
\caption{Euler characteristic $\mathbf E({\mathcal L},x,y)$.
Circle at $(a,b)\in \Z^2$ is green if the respective coefficient is positive, blue --- otherwise.
The area of circles is proportional to the coefficients, indicated also by numbers.
The shaded region is the strip for monomials of $\mathcal L$.
The red lines pass through the pivot elements.
The last red line passes through $v_{10}$.}
\label{Fig4} 
\includegraphics [width=0.95\textwidth]{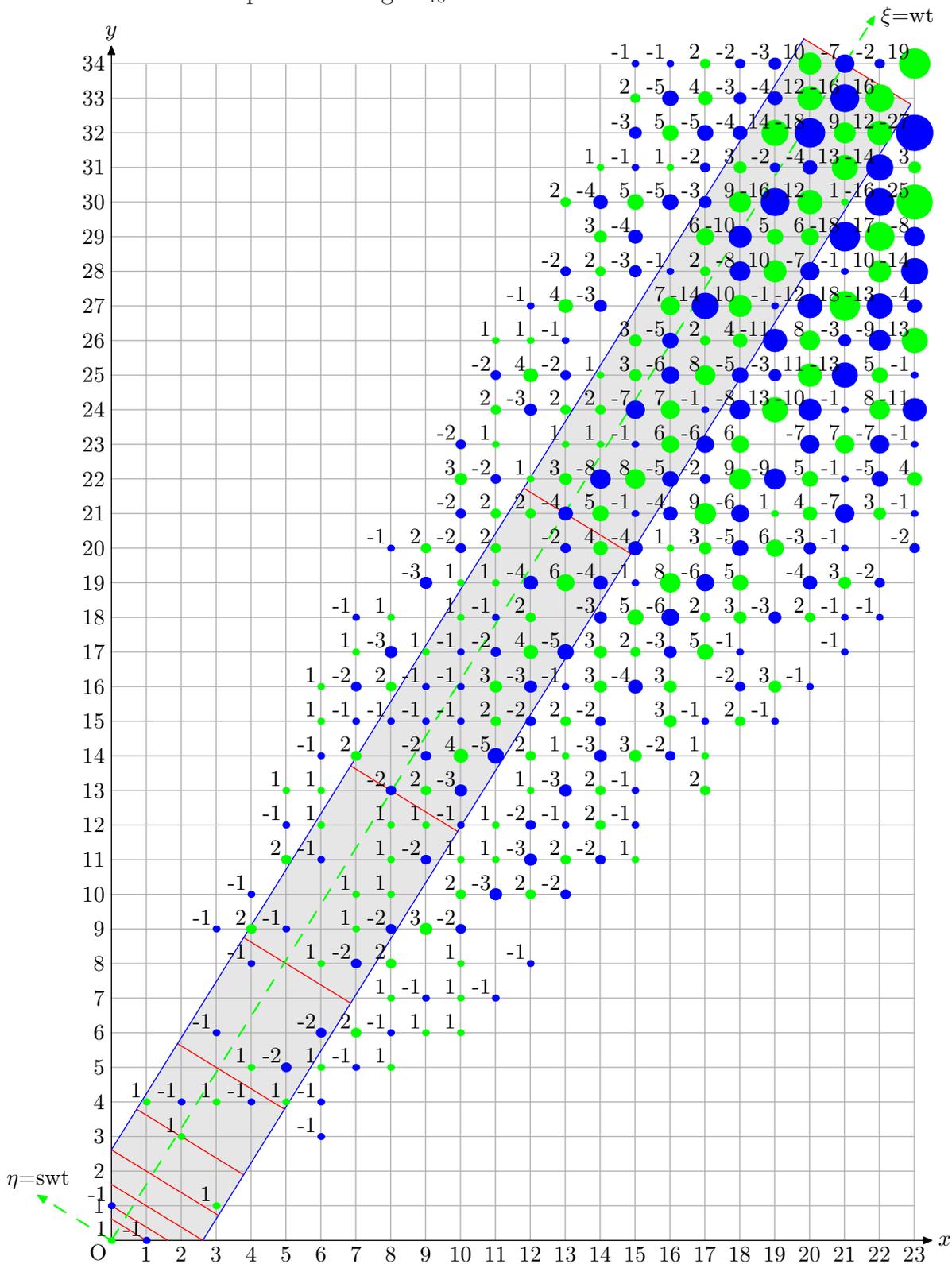}
\end{figure}
%[scale=0.9]

%%\bibliography{sn-bibliography}

\end{document}